\documentclass[fleqn]{amsart}
\usepackage{amsmath}
\usepackage{amssymb, amstext, amsthm, amsfonts,euscript,amscd,mathrsfs}

\usepackage[colorlinks=false]{hyperref}

\usepackage[left=3.2cm,top=3cm,right=3.3cm,bottom=1.4in,asymmetric]{geometry}

\DeclareMathOperator{\Real}{\mathbb{R}}

\newcommand{\jap}[1]{\!\left<#1\right>}

\newcommand{\dx}{\partial_x}

\newcommand{\dxi}{\partial_{\xi}}

\providecommand{\norm}[1]{\lVert#1\rVert}
\providecommand{\abs}[1]{\lvert#1\rvert}

\newcommand{\conv}{\ast}
\newcommand{\pd}{\partial}

\newcommand{\dt}{\pd_t}
 \newcommand{\linf}{L^{\infty}}
\newcommand{\tld}[1]{\tilde{#1}}

\newcommand{\R}{\Real}
\newcommand{\vu}{{\bf u}}
\newcommand{\eps}{\varepsilon}

  \newcommand{\vue}{{\bf u^\eps}}
  \newcommand{\ue}{u_{\eps,\delta}}
  \newcommand{\uee}{u_{\eps'}}

\newcommand{\les}{\lesssim}

\newcommand{\vz}{\vec{z}}
\newcommand{\dz}{\partial_{\vz}}

  \newcommand{\vo}{\vec{0}}

\newcommand{\Schw}{\mathscr{S}}

  \newcommand{\K}{\mathcal{K}}

\newcounter{theorem_counter}
\newtheorem{thm}{Theorem}
\newtheorem{lem}[thm]{Lemma}
\newtheorem{prop}[thm]{Proposition}
\newtheorem{rem}[thm]{Remark}
\newtheorem{coro}[thm]{Corollary}

\newtheorem{defn}[thm]{Definition}

\title{Well-posedness of fully nonlinear KdV-type evolution equations}
\author{Timur Akhunov}
\address{The College of New Jersey, Department of Mathematics and Statistics, Trenton, NJ, USA}
\email{akhunovt@tcnj.edu}
\author{David M. Ambrose}
\address{Drexel University, Department of Mathematics, Philadelphia, PA, USA}
\email{dma68@drexel.edu}
\author{J. Douglas Wright}
\address{Drexel University, Department of Mathematics, Philadelphia, PA, USA}
\email{jdoug@math.drexel.edu}

\begin{document}

\maketitle

\begin{abstract}We study the well-posedness of the initial value problem for fully nonlinear evolution equations,
$u_{t}=f[u],$ where $f$ may depend on up to the first three spatial derivatives of $u.$  We make three primary
assumptions about the form of $f:$ a regularity assumption, a dispersivity assumption, and an assumption related to the
strength of backwards diffusion.  Because the third derivative of $u$ is present in the right-hand side and we effectively
assume that the equation is dispersive, we say that these fully nonlinear evolution equations are of KdV-type.
We prove the well-posedness of the initial value problem in the Sobolev space $H^{7}(\mathbb{R}).$
The proof relies on gauged energy estimates which follow after making two regularizations, a parabolic regularization
and mollification of the initial data.
\end{abstract}

\section{Introduction}

We study the question of well-posedness in Sobolev spaces of the initial value problem for
the fully nonlinear evolution equation
\begin{equation}\label{generalEquation}
u_{t}=f(u_{xxx},u_{xx},u_x,u,x,t),
\end{equation}
under suitable conditions on the function $f$ and its partial derivatives.
Chief among these conditions will be a condition which ensures that the
equation is dispersive, so that the contribution of $u_{xxx}$ to the function
$f$ is in a sense dominant and nonvanishing.  An explicit calculation in
Fourier space shows that the equation $u_{t}=u_{xxx}-u_{xx}$ is ill-posed in $L^2$-based Sobolev spaces
because of the presence of backwards diffusion; another principal condition
for well-posedness therefore must be a balance between
the effects of dispersion and backwards diffusion.

There are several papers treating existence theory for semilinear dispersive
equations, especially the case in which the leading-order term has a constant
coefficient.  Kenig, Ponce, and Vega show local well-posedness of the initial
value problem for
\begin{equation}\label{kpvEquation}
u_{t}+\partial_{x}^{2j+1}u+P(\partial^{2j+1}_{x}u,\ldots,\partial_{x}u,u)=0,
\end{equation}
with $P$ a polynomial and $j\in\mathbb{N},$ using weighted Sobolev spaces
\cite{kpv1994}.
Kenig and Staffilani treated the generalization of \eqref{kpvEquation} to systems,
again proving local well-posedness in weighted Sobolev spaces \cite{KenStaf97}.
Kenig and Pilod  in \cite{kenig-pilod2} studied some special cases of \eqref{kpvEquation}
which include the integrable KdV hierarchy.
Recently, Harrop-Griffiths has treated the $j=1$ case of \eqref{kpvEquation}.
In \cite{harrop-griffiths1}, local well-posedness is proved in certain translation-invariant
subspaces of Sobolev spaces on the real line.  Under a further assumption on the
polynomial $F,$ Harrop-Griffiths also proves well-posedness in Sobolev spaces
\cite{harrop-griffiths2}.  In another recent work, Germain, Harrop-Griffiths, and Marzuola
prove existence of spatially localized solutions for a particular quasilinear KdV-type equation
\cite{newMarzuolaEtAl}.

As mentioned briefly above, all
of these results must contend with the fact that in some cases, equations of
the form \eqref{generalEquation} or \eqref{kpvEquation} can be ill-posed;
ill-posedness results have been proved in \cite{Pilod08}, \cite{Akh12},
\cite{AmWr12}.
The choice of spaces other than the $L^{2}$-based Sobolev spaces $H^{s}$
in \cite{kpv1994}, \cite{KenStaf97}, \cite{harrop-griffiths1} allows the ill-posedness
to be avoided.  Alternatively, in \cite{harrop-griffiths2}, the additional condition on
the polynomial $F$ (that there is no term of the form $uu_{xx}$) removes the
ill-posedness.

In the non-constant coefficient, semilinear case, Cai shows dispersive smoothing properties
for solutions of $u_{t}-a(x,t)\partial_{x}^{3}u+P(x,t,\partial_{x}^{2}u,\partial_{x}u,u)=0$ \cite{cai}.
Here, the coefficient $a$ is required to be bounded away from zero, so that the dispersion does
not vanish.  The result of \cite{cai} is related to results of \cite{CKS92}; there,
Craig, Kappeler, and Strauss proved well-posedness and dispersive smoothing for solutions of
\eqref{generalEquation}, under some fairly strong assumptions.  Both \cite{cai} and
\cite{CKS92} use weights to ensure certain rates of decay at infinity; in the semilinear
case, Cai is able to weaken the assumptions of Craig, Kappeler, and Strauss.
The assumptions of these papers, as in the present work, include a condition that controls the effect
of backward diffusion.  In the present work, as in \cite{CKS92} and unlike \cite{cai},
we treat the fully nonlinear evolution equation \eqref{generalEquation}; similarly to \cite{cai},
we are interested in significantly weakening the conditions imposed in \cite{CKS92}.
Further differentiating our work from \cite{CKS92}, we do not use weights, and instead
only use the $L^{2}$-based Sobolev spaces $H^{s}.$

The authors have previously established well-posedness results for
 initial value problems in some special cases of the equation
\eqref{generalEquation}, including some quasilinear equations.
Akhunov has shown well-posedness of quasilinear systems  \cite{Akh13}
and linear equations \cite{Akh12} on the real line.
In \cite{AmWr12}, Ambrose and Wright studied linear
equations on periodic intervals, as well as certain specific quasilinear equations such as the $K(2,2)$
Rosenau-Hyman compacton equation \cite{Ros-Hym93} and the Harry Dym equation \cite{kruskal}.
The results of the present paper are given on the real line; the authors may treat the spatially periodic
case in a future work.

The conditions that we presently require on $f$ are similar to the conditions assumed by the authors in
\cite{Akh12} and \cite{AmWr12}, adapted to the fully nonlinear evolution
equation \eqref{generalEquation}, and allowing for as much generality as possible.  These
conditions will be specified more technically in what follows, but they are, roughly: (a) the function $f$
is sufficiently smooth, (b) the partial derivative of $f$ with respect to $u_{xxx}$ does not vanish, so that
the dispersion does not degenerate, and (c) the ``modified diffusion ratio,'' to be defined, but which
balances the effects of dispersion and backwards diffusion, must either be integrable or be the derivative
of a smooth function (this condition is closely related to the ``Mizohata" condition needed in \cite{harrop-griffiths1}).
With these conditions satisfied, we are able to use a gauged energy estimate and
a parabolic regularization to prove well-posedness of initial value problems in Sobolev spaces.

Allowing for the most general $f$ possible requires us to study solutions at somewhat high regularity, $H^{7}.$
We discuss certain special cases in which we are able to lower this regularity requirement, such as the
case of quasilinear equations.  We mention that by disallowing the occurrence of terms of the form $uu_{xx}$
in the nonlinearity, Harrop-Griffiths was able to use Sobolev spaces $H^{s}$ for $s>\frac{9}{2}$ for
semilinear equations;
furthermore, Harrop-Griffiths includes a discussion of when lower regularity results are possible, depending
on whether certain terms are or are not present in the nonlinearity.
In \cite{AmWr12}, in the spatially periodic case,
well-posedness of the initial value problem with strictly positive data for the $K(2,2)$ equation
$u_{t}+(u^{2})_{xxx}+(u^{2})_{x}=0,$ among other quasilinear problems, was demonstrated in $H^{4}.$

The plan of the paper is as follows:  in Section 2, we state the main result.  In Section 3, we introduce a regularized
problem.  In Section 4, we prove a useful estimate for a related linear evolution equation.  In Section 5, we use this
linear estimate to find bounds for the solution of the nonlinear regularized problem.  In Section 6, we pass to the limit,
finding solutions to the original, non-regularized nonlinear problem.  We close with some discussion in Section 7, and as an appendix we provide an explicit calculation of the third spatial derivative
of \eqref{generalEquation}.

D.M. Ambrose is grateful to the National Science Foundation for support through grant DMS-1515849.
J.D. Wright is grateful to the National Science Foundation for support through grant DMS-1511488.

   \section{Statement of the result}
       Suppose that the following assumption on the function $f(\vz,x,t)$ from \eqref{generalEquation} holds:
       \begin{description}
         \item[Condition (A1)] Regularity. Assume $f(\vz,x,t) \in C^1_t C^{11}_{\vz}W^{11,\infty}_x$. \\
       That is, the nonlinear function can grow in the dependent variable $\vz$ and is bounded in the independent space-time variables. Stated more precisely, for any $k\ge 0$ and $|(z_0,\ldots z_3)|\le k$, we have for each choice
       of $\alpha,$ $\beta,$ and $\gamma$ satisfying $\alpha \le 1$, $\max \{\abs{\beta},\gamma\} \le 11,$ there exists a positive increasing function $s \mapsto C_{\alpha,\beta,\gamma}(s)$ such that
       \begin{align}\label{coeff-bdd}
         \norm{\dt^\alpha \partial_{z}^\beta\dx^\gamma f(\vz,x,t)}_{L^\infty_{x,[-1,1]}}\le C_{\alpha,\beta,\gamma}(\abs{\vz})\le C(k).
       \end{align}

       \end{description}
       \begin{rem}\label{rem:dx-u}
         To simplify the accounting of functions depending on derivatives of $u$, we introduce the following notation:
$$
a(\dx^k \vu):=a(\dx^k u,\ldots,\dx u, u,x,t)
$$
i.e. we bold the highest derivative on which the function in question depends and hide all other factors. Often we also introduce a variable $\vz(x,t)=(\dx^k u,\ldots,\dx u,u,x,t)$ for the same purpose.

Furthermore, we define
$f_{z_i}=\partial_{z_{i}} f(z_3,\ldots z_0,x,t)$ to be derivatives with respect to the derivatives of the unknown solution $u$.
       \end{rem}
           \begin{description}
           \item[Condition (A2)] Suppose that the ``modified diffusion ratio" $g_M$, defined below, has the following form
            \begin{align}\label{Diffusion}
           g_M:=\frac{f_{z_2}}{f_{z_3}}(\dx^3\vu)= \dx [ g_D(\dx^2 \vu)] + g_H(\dx^3\vu),
         \end{align}
         where $g_D \in C^1_t C^{0}_z L^\infty_{x,z}$ and $g_H \in C^1_t C^2_{z}L^\infty_x$, i.e. satisfying bounds similar to \eqref{coeff-bdd} with lower regularity. Moreover,
         \begin{align}\label{g_H-is-cubic}
           g_H(\vo,x,t) \equiv 0 \equiv\partial_{z_j} g_H(\vo,x,t) \text{ for } j=0,\ldots 3.
         \end{align}
       \item[Condition (A3)] Assume $f(\vo,x,t)=0$.
         \end{description}
Before stating the next result, we comment on the meaning of these conditions.

\begin{rem}\label{coefficients-simplicity}
  Unlike (A1) and (A2), condition (A3) is done for simplicity. All the arguments below are valid, but are lengthier in the presence of an additional forcing term $f(\vo,x,t)\neq 0$ that is independent of the solution.\\

  Likewise, \eqref{Diffusion} can be thought of as Taylor expansion of $\frac{f_{z_2}}{f_{z_3}}$ with respect to a vector $\dx^3\vu$, with additional assumptions. By ruling out degeneracy of the dispersion coefficient $f_{z_3}$ (to be justified below), we observe that $\frac{f_{z_2}}{f_{z_3}}$ is a smooth function and hence the most general Taylor expansion is
  \begin{align*}
    \frac{f_{z_2}}{f_{z_3}} = g_C(x,t) + \sum_{j=0}^3 g_j(x,t)\dx^j u + g_H(x,t,\dx^3 \vu),
  \end{align*}
  where $g_H$ has no constant or linear terms in $\dx^3\vu$, like $g_H$ in \eqref{g_H-is-cubic}. The technical framework of the paper allows the $g_C$ term, as in \cite{Akh12}. We choose to omit it for simplicity. However, the argument breaks down unless the ``linear coefficient" terms $g_j\dx^j u$ have the total derivative form of \eqref{Diffusion}. We address the necessity of this below, after stating our main theorem.
\end{rem}
We define
       \begin{align}
         \label{Dispersion}
         \lambda(\dx^3\vu(t)) := \left\| \frac{1}{ {f_{z_3}}(\dx^3\vu(t))} \right\|_{L^\infty_{x}}
       \end{align}

         \begin{rem}\label{Dispersion-nondeg}
         Note that the quantity $\lambda$ measures non-degeneracy of the dispersion in \eqref{generalEquation}. In the case when $\lambda(\dx^3 \vu_0)>0$, we demonstrate in section 5 that this condition remains valid for a small time determined by the size of the solution.
\end{rem}
       \begin{thm}\label{main-thm}
         Suppose $f$ from \eqref{generalEquation} satisfies (A1)--(A3). Let $u_0 \in H^7$, be such that
         \begin{align}\label{Dispersion-data}
           \lambda_0:=\lambda(\dx^3 \vu_0)<\infty \text{  for } \lambda\text{ from }\eqref{Dispersion}.
         \end{align}
         Then there exists $T=T(\norm{u_0}_{H^7},\lambda_0)$, such that \eqref{generalEquation} is wellposed. That is
         \begin{itemize}
           \item There exists a classical solution $u\in C_{[-T,T]}H^{7}$ of \eqref{generalEquation},
           \item This solution $u$ is unique in the class $C_{[-T,T]}H^{7},$ and
           \item The solution $u$ depends on data $u_0\in H^7$ continuously. That is, if $u_{0,n}\to u_0\in H^7$, then the associated solutions $u_n$ satisfy  $u_n \to u \in   C_{[-T,T]}H^{7}$
         \end{itemize}
       \end{thm}
 We return to discuss Remark \ref{coefficients-simplicity} and the sharpness of Theorem \ref{main-thm}. The condition \eqref{Diffusion} on the ``modified diffusion ratio'' was demonstrated \cite{Akh12} to be necessary for the wellposedness of the linear case of \eqref{generalEquation}. A condition of the form \eqref{Diffusion} is likely necessary for the wellposedness in $H^s$. Pilod \cite{Pilod08} has demonstrated that for the evolution equation $\dt u + \dx^3 u + u\dx^2 u=0,$ the flow map
$u_0 \mapsto u$ on $H^{s}$ is not $C^2;$ note that for this equation, $g_M = u$, which is not of the form \eqref{Diffusion}. In future work, we expect to extend these illposedness results to show a lack of continuous dependence on data and hence demonstrate the sharpness of Theorem \ref{main-thm} with respect to the ``modified diffusion ratio."
\subsection{Notation}
When estimating with multiplicative constants, we will often write $A\les_{x,y} B$, to mean $A \le C\cdot B$, where the constant $C= C(x,y)$ may increase from line to line. By an equivalence $A\approx_{x,y} B$, we mean $A \les_{x,y} B \les_{x,y} A$. In most of this work, the constants $C$ will depend on the nonlinear function $f$, i.e. $C=C(f, s,k)$. The functional dependence of the constants on dispersion and the size of data is of paramount importance and will be kept, e.g. $C=C(\lambda(t),\norm{u(t)}_{H^7})$.
\subsection*{Sobolev Spaces}
We will use $L^{2}$-based Sobolev spaces extensively and will define them here for reference. In particular, by $H^s$ we mean the set of tempered distributions $f$, such that $$\norm{f}_{H^s}:=\norm{(1+\abs{\xi}^2)^{\frac{s}{2}} \hat f(\xi)}_{L^2} <\infty.$$ When $s$ is a non-negative integer the Fourier Transform turns derivatives into multiplication, hence $$\norm{f}_{H^s} \approx \norm{f}_{L^2}+\norm{\dx^s f(x)}_{L^2}.$$ 

For $s>\frac{1}{2}$ we often choose a version of $f \in H^s(\R)$ that is in addition a smooth function in $C^{\lfloor s-\frac{1}{2}\rfloor}(\R)$, where $\lfloor \_ \rfloor$ is a lower integer part of a number. Choosing such a version is well-defined by the Sobolev embedding.
\subsection*{Space-time norms}\label{subsec:norms}
  We will use the following space-time norms:
\begin{align*}
& \norm{f(x,t)}_{L_{[t_1,t_2]}^{\infty}L_x^2} \equiv \sup_{t\in {[t_1,t_2]}}
\left(\int_{\R} \abs{f(x,t)}^2dx\right)^{\frac{1}{2}},\\
& \norm{f(x,t)}_{L^1_{[t_1,t_2]} L^2_x} \equiv \int_{[t_1,t_2]}\left(\int_{\R} \abs{f(x,t)}^2dx\right)^{\frac{1}{2}}dt.
\end{align*}
Most of the time $[t_1,t_2]$ will stand for $[0,T]$ or $[-T,0]$ or $[-T,T]$ for some $T>0$.

     \section{Regularizations}

  To prove wellposedness, we will ues two types of regularization -- one on the data, and one on the equation.
  We first regularize the data, then the equation. The solution of \eqref{generalEquation} will be constructed as a limit of such regularized solutions.
\subsection{Data regularization}
  Let $0\le \phi(\xi) \le 1$ be a radial smooth bump function satisfying
     \[
     \phi(\abs{\xi})=      \begin{cases}
      & 1,\text{ if }\abs{\xi} \le 1,\\
       & 0,\text{ if }\abs{\xi} \ge 2.
     \end{cases}
     \]
     For all $0<\delta\le 1$, define $\mathcal{F}\left( (u_{0})_{\delta}\right)(\xi) = \hat u_0(\xi)\cdot \phi(\delta\abs{\xi})$.
     Hence
     \begin{align}\label{regularizationDefinition}
       (u_{0})_{\delta}=\left[u_0\conv \frac{1}{\delta}\check{\phi}\left(\frac{1}{\delta}\cdot\right)\right](x).
     \end{align}

     The structure of $\phi$ allows us to give quantitative bounds on the convergence of $(u_{0})_{\delta}$
     to $u_0$ in $L^2$ and in $H^7$ as ${\delta\to 0}$. To do this, we restate Lemma 4.1 from \cite{Akh13}, as we need the details in the present manuscript.
     \begin{lem}\label{lem:regularize}
   \label{BS} Let $\K \Subset H^{7}$ be a compact set. Then $\forall \delta>0$, and any $u_0 \in \K$, $(u_{0})_{\delta} \in \Schw$ defined above satisfies
   \begin{subequations}
     \begin{align}
      \norm{(u_{0})_{\delta}}_{H^{7+j}} \le C_j\norm{u_0}_{H^7}\delta^{-j},\quad       \text{ for all } j\ge 0,\,\,\,  \label{eq:BS:above}\\
     \norm{(u_{0})_{\delta}-u_0}_{L^2} = o(\delta^{7})\quad \text{ and }\quad
      \norm{(u_{0})_{\delta}-u_0}_{H^{7}} = o(1),\label{eq:BS:L2}
   \end{align}
   \end{subequations}
   with the convergence rate dependent on $\K$.
 \end{lem}
\begin{proof}
     Let $j\ge 0$ and $0<\delta <1$ be given.  We calculate as follows:
     \begin{align}\label{BSP1}
     \begin{split}
       & \norm{(u_{0})_{\delta}}_{H^{7+j}}^2  = \int_{\abs{\xi} \le \frac{2}{\delta}} (1+\abs{\xi}^2)^{7}\abs{\hat u_0(\xi) }^2 (1+\abs{\xi}^2)^j \phi(\delta\abs{\xi})^2 d\xi \le  \left(\frac{3}{\delta}\right)^{2j} \norm{u_0}_{H^{7}}^2.
     \end{split}
     \end{align}
    This proves \eqref{eq:BS:above}.
      For \eqref{eq:BS:L2} it suffices to show the first estimate, with the second done identically. By the Fundamental Theorem of Calculus, we have
     \begin{align*}
      & \norm{(u_{0})_{\delta}-u_0}_{L^2}^2 = \int (\phi(\delta\abs{\xi})-1)^2\abs{\hat u_0(\xi)}^2 d\xi \le \int \left(\sup_{\eta \in B_{\delta\abs{\xi}}(0)}\abs{\phi'(\eta)}\right)^2 \delta^2\abs{\xi\hat u_0}^2 d\xi,
     \end{align*}
     where we have
     used $\phi(0)=1$ and defined $B_{\delta\abs{\xi}}(0) = [-\delta\abs{\xi},\delta\abs{\xi}]$. As $\dxi^j\phi(0)=0$ for all $j>0$, we can continue with the Taylor expansion of $\abs{\phi^{(j)}(\eta)}$ seven times and then use $\dxi^j\phi\equiv 0$ on $B_1(0)$ for $j>0$ to conclude
     \begin{align*}
       & \norm{(u_{0})_{\delta}-u_0}_{L^2}^2  \le \delta^{14} \norm{\dxi^{7} \phi}_{\linf}^2\int_{\abs{\xi}\ge \frac{1}{\delta}} \abs{\xi^{7}\hat u_0}^2 d\xi.
     \end{align*}
The $o(1)$ rate that is uniform for $u_0$ in a compact set $\K$ comes by the Lebesgue Dominated Convergence Theorem.
\end{proof}
\subsection{Parabolic regularization of the main equation}
 We now add both the data regularization and the parabolic regularization to \eqref{generalEquation}:
  \begin{align}\label{parab:non}
  \begin{cases}
    \dt\ue + f(\dx^3\ue,\ldots,\ue, x,t) = -\eps\dx^4 \ue,\\
    \ue(x,0)=(u_{0})_{\delta}(x).
  \end{cases}
\end{align}
Note that for $\eps = 0$, this equation is \eqref{generalEquation}. Whereas for $\eps>0$, this equation is a semilinear parabolic equation and is well-understood. We quote the following wellposedness result.
    \begin{prop}\label{parab:prop}
      Suppose $(u_{0})_{\delta}\in H^s$ for $s\ge 4$ and let $\eps>0$ be given. There exists a maximal interval of existence $[0,T_\eps)$, such that  $T_\eps \ge 1/C(\frac{1}{\eps},\norm{u_{0,\eps}}_{H^4},s)$, so that \eqref{parab:non} is well-posed in $H^s.$ In particular, \eqref{parab:non} has a unique solution in $C_{[0,T_\eps)}H^s\cap C^1_{[0,T_\eps)} H^{s-4}\cap L^2_{[0,T_\eps)}H^{s+2}_x$ and this solution $\ue$ depends continuously on data, i.e. $((u_{0})_{\delta})\mapsto \ue(t)$ is a continuous map from $H^s\to H^s$. Moreover, if the maximal time of existence satisfies $T_\eps <\infty$, then \eqref{parab:non} has the following blow up criterion:
      \begin{align}\label{parabolic-blow-up}
        \lim_{t\nearrow T_\eps} \norm{\ue(t)}_{H^4} = \infty.
      \end{align}
    \end{prop}
    \begin{proof}[Sketch of proof]
     The proof of the proposition follow the outline of a standard semilinear parabolic problem, e.g. \cite{Akh13}. The linear semi-group gains $3$ derivatives in $L^2$-based spaces and allows the $f(\dx^3\vue)$ terms to be treated as a lower order perturbation for a contraction mapping argument.
    \end{proof}

\subsection{Main functional norms}
We define
  \begin{align}
    M_\eps(t) = \norm{\ue}_{L^\infty_t H^7}+\eps\norm{\ue}_{L^\infty_t H^8},\label{high-norm}\\
    k(t):= \sup_{0\le t'\le t}\left\{ \norm{{\ue}(t')}_{H^4};\lambda(\dx^3{\bf\ue}(t'))\right\}. \label{dispersion+low-norm}
  \end{align}
  The main technical ingredient of the paper is the simultaneous control of the high Sobolev norm $M_\eps$ and this function $k(t)$, which measures dispersion and low Sobolev norms of the solution.
\begin{rem}
  \label{M0}
  Note that upon combining Lemma \ref{BS} with the definition of $M_\eps(t)$ we obtain
  \begin{align*}
    M_\eps(0)\le \norm{u_0}_{H^7}(1+ \eps C\delta^{-1}).
  \end{align*}
  Thus we can ensure that there exists $0<\delta_0\ll 1$ such that
  \begin{align}\label{u0-M}
    M_\eps(0) \approx \norm{u_0}_{H^7} \text{ for } \eps \le \delta^2 \text{, for $\delta\le \delta_0.$}
  \end{align}
\end{rem}
  \section{A Linear Estimate}

  In this section, we develop an estimate for solutions of a linear problem.  In the subsequent sections, we will
  prove estimates for the nonlinear problem by applying this linear estimate.  The linear problem we consider now is
  as follows:

     \begin{align}\label{lin}
\begin{cases}
  \dt  w + \sum_{j=0}^3 a_j(x,t)\dx^{j} w =  -\eps \dx^4 w + h,\\
  w(0,x) \equiv w_0(x).
\end{cases}
\end{align}

 Let    \begin{align}
  \begin{split}
  & k_G(t): = \left\|\int_0^x \frac{a_2}{a_3}(x',t)\ dx'\right\|_{L^\infty_x}
  +\left\|\frac{1}{a_3}(t)\right\|_{L^\infty_x}
  +\left\|a_3(t)\right\|_{L^\infty_{x}},\\
    \label{linear-coeff}
    & \tld M(t): = \sup_{0\le t'\le t}\left(\sum_{i=0}^3 \norm{a_i(t')}_{L^\infty_T W^{i,\infty}_x}
    +\left\|\frac{1}{a_3(t')}\right\|_{L^\infty_{x}}
    + \left\|\int_0^x \frac{a_2(y,t)}{{a_3(y,t)}}\ dy\right\|_{L^\infty_{x}}\right)\\
 & +\sup_{0\le t'\le t}\left(
 \left\|\dt \int_0^x \frac{a_2(y,t')}{{a_3(y,t')}}\ dy\right\|_{L^\infty_{x}}
 + \norm{\dt a_3(t')}_{L^\infty_{x}}\right).
  \end{split}
\end{align}

The following theorem is the main result of the present section:
    \begin{thm}\label{thm:lin}
    Assume $k_G(t)$ and $\tld M(t)$ are bounded for $t\in [0,T]$ for some $T>0$, and $h\in L^1_tL^2_x$. Then any classical solution $w$ of \eqref{lin} with $w\in C^1_t H^{-2} \cap C^0_t H^3$ satisfies
      \begin{align}\label{lin:ref}
        \norm{w(t)}_{L^2}+\eps\norm{w}_{L^2_t H^2_x}
        \le C(k_G(t))\exp\left(\int_0^t C(\tld M(t')) dt'\right)\left( \norm{w_0}_{L^2} + \norm{h}_{L^1_t L^2}\right).
      \end{align}

    \end{thm}
    In \cite{Akh12} the first author established a similar linear estimate with a constant $C(\norm{\tld M(t)}_{L^\infty_t})$ without the exponential. However, the $H^7$ wellposedness we are proving in Theorem \ref{main-thm}, as compared to
    the $H^{12}$ result
    of \cite{Akh12}, demands more delicate accounting of constants than the one given in \cite{Akh12}. In particular, constant dependence on $\ k_G(t)$ and $\tld M(t)$ is done separately, which was not the case in \cite{Akh12}. We also confirm that the additional $-\dx^4$ term is harmless for energy estimates.\\

 The proof is based on the energy method, attempting to estimate $\dt \norm{w}_{L^2}^2$ by $\norm{w}_{L^2}^2$. As this method does not apply directly we modify the solution $w$ first.
    \begin{defn}\label{gauge}
  A function $\phi\in C^0_{[0,T]}W^{3,\infty}_x\cap C^1_{[0,T]}L^\infty$ is called a \textbf{gauge}, if the following bounds hold:
  \begin{align}
     & \norm{\phi(t)}_{L^\infty_x}+\left\|{\frac{1}{\phi(t)}}\right\|_{L^\infty_x} \le C(k_G(t)),\label{gauge-below}\\
  &\norm{\phi}_{W^{3,\infty}}+ \norm{\dt \phi}_{L^\infty_T} \le C(\tld M(t)),\label{guage-der}
  \end{align}
 with $k_G(t)$ and $\tld M(t)$ as defined in \eqref{linear-coeff}.
\end{defn}
     Given a gauge, $\phi,$ we define
\begin{align}\label{gauged:var}
 v= \phi^{-1} w.
\end{align}
A substitution of $v$ into \eqref{lin} gives:
\begin{align}\label{vt}
  \begin{cases}
    \dt v + L_\phi v = \phi^{-1}h -\eps \phi^{-1}\dx^4(\phi v),\\%
    v(x,0)= \phi^{-1}u_0,
  \end{cases}
\end{align}
where
\begin{multline}\label{LphiDefn}
L_\phi = a_3\dx^3 +\left(a_2 + \phi^{-1}{3 a_3\dx \phi}\right)\dx^2  +\left(a_1+ \phi^{-1}({2 a_2\dx \phi+ 3 a_3\dx^2\phi})\right)\dx\\
  + \left(a_0 +{\phi^{-1}}({\dt \phi+a_1\dx \phi+a_2\dx^2\phi+a_3\dx^3 \phi})\right).
\end{multline}

\begin{lem}\label{comparability}
Let $\alpha$ and $\beta$ be related by \eqref{gauged:var}, e.g. $\alpha=\phi^{-1}\beta$. Then $\alpha \in C^1_{[0,T]}H^{-2}\cap C^0_{[0,T]}H^{2}$ if and only if $\beta \in C^1_{[0,T]}H^{-2}\cap C^0_{[0,T]}H^{2}$ with comparability constants of the form $C(\tld M(t))$. Moreover,
\begin{align}\label{v}
  \norm{\beta}_{L^\infty_T L^2_x} \approx_{k_G(t)} \norm{\alpha}_{L^\infty_T L^2_x}.
\end{align}
\end{lem}
\begin{proof}
 From \eqref{gauged:var} and \eqref{gauge-below} we immediately conclude \eqref{v}. Similarly, differentiating \eqref{gauged:var} twice, we find
  \begin{align*}
    \norm{\beta}_{H^2}\approx_{\norm{\phi}_{W^{2,\infty}},\norm{\frac{1}{\phi}}_{L^\infty}} \norm{\alpha}_{H^2}.
  \end{align*}
   We use these observations to justify comparability of norms for $H^{-2}$ using duality, where we apply the estimate above to a test function $\gamma=\phi^{-1}\cdot (\phi\gamma)$:
  \begin{equation*}
    \norm{\beta}_{H^{-2}} = \sup_{\norm{\gamma}_{H^2}\le 1}
    \left| \int \phi \alpha \cdot \gamma \right|
    \le \sup_{\norm{\gamma}_{H^2}\le 1} \norm{\alpha}_{H^{-2}}C(\norm{\phi}_{W^{2,\infty}})\norm{\gamma}_{H^2} \les_{\norm{\phi}_{W^{2,\infty}}} \norm{\alpha}_{H^{-2}}.
  \end{equation*}
  Similarly, $$\norm{\alpha}_{H^{-2}}\les_{\norm{\phi}_{W^{2,\infty}},\norm{\frac{1}{\phi}}} \norm{\beta}_{H^{-2}}.$$
  Finally, using the fact that $L^2\subseteq H^{-2},$ we have
  \begin{multline}\nonumber
     \norm{\dt \beta}_{H^{-2}} \le \norm{\dt \phi \cdot \alpha}_{H^{-2}} + \norm{\phi \cdot \dt \alpha}_{H^{-2}}\\
     \le \norm{\dt \phi \cdot \alpha}_{L^2} + \norm{\phi \cdot \dt \alpha}_{H^{-2}} %
     \le C(\tld M(t))(\norm{\alpha}_{L^2}+\norm{\dt \alpha}_{H^{-2}}).
  \end{multline}
  Hence $\norm{\dt \beta}_{H^{-2}}+\norm{\beta}_{H^2} \les_{\tld M(t)}\norm{\dt \alpha}_{H^{-2}}+\norm{\alpha}_{H^2}$ and similarly for the other comparability direction.
\end{proof}
\begin{rem}\label{rem:gauge}
Observe, that by Lemma \ref{comparability} applied to $v$ and $w$ from \eqref{gauged:var}, the proof of Theorem \ref{thm:lin} is reduced to \eqref{lin:ref} for $v$ satisfying \eqref{vt}.
\end{rem}

The energy method involves multiplying \eqref{vt} by $v$ to estimate $\dt \norm{v}_{L^2}^2$ by $\norm{v}_{L^2}^2.$
We begin as follows:
\begin{align}\label{v:energy}
  \dt \int \abs{v}^2 = - 2 ( L_\phi v,v) - ( \eps \dx^4(\phi v)-h, \phi^{-1}v),
\end{align}
    where $(u,v)$ is an $L^2_x$ pairing. We quote the following integration by parts argument:

    \begin{lem}\label{enrg}[Lemma 3.5 from \cite{Akh12}:]
  Consider an operator $L=b_3\dx^3 + b_2\dx^2 + b_1\dx +b_0$, where $b_j \in L^\infty_{[0,t]} W^{j,\infty}_x$. Define $c_0=b_0-\frac{1}{2}(\dx b_1-\dx^2 b_2 +\dx^3 b_3)$. Then for any $v\in C^0_{[0,T]}H^{3}$
  \begin{align*}
     (Lv,v) = \left(\left[-b_2+\frac{3}{2} \dx b_3\right]\dx v, \dx v\right) + (c_0 v,v)
  \end{align*}
\end{lem}

As can be seen in \eqref{LphiDefn}, the $\dx^{2}$ coefficient in $L_{\phi}$ includes both $a_{2}$ as well as a term involving
$\phi.$
Hence applying Lemma \ref{enrg} to \eqref{v:energy}, we see that a choice of $\phi$ can be made to eliminate derivative terms in $(L_\phi v,v)$. In particular, applying Lemma \ref{enrg} to \eqref{v:energy}, the choice of $\phi$ we need is the one to satisfy the following identity: $$  \left(\frac{1}{2}\left[2 a_2 + \frac{6 a_3\dx \phi}{\phi}- 3 \dx a_3\right]\dx v, \dx v\right)=0.$$The Lemma below justifies that such a choice of $\phi$ is indeed a gauge.
\begin{lem}
  \label{gauCh}
  Let $\phi(x,t)$ be a solution of the ODE
  \begin{align}\label{phi:ODE}
    \begin{cases}
      6 a_3 \dx \phi = \left( 3 \dx a_3  - 2a_2\right) \phi,\\
      \phi(t,0) = 1.
    \end{cases}
  \end{align}
Then $\phi$ is a gauge in the sense of the Definition \ref{gauge}.
\end{lem}
\begin{proof}
  The ODE for $\phi$ is solved explicitly as
  \begin{align*}
    \phi(x,t) = \sqrt{\frac{a_3(x,t)}{a_3(0,t)}}e^{-\int_0^x \frac{a_2(y,t)}{3a_3(y,t)}dy}.
  \end{align*}
  From this definition \eqref{gauge-below} follows. Differentiation of $\phi$ using \eqref{phi:ODE} yields \eqref{guage-der}.
\end{proof}
We consider the parabolic term in \eqref{v:energy} and demonstrate that a change of variables from the Definition \ref{gauge} does not significantly affect it.
\begin{lem}\label{energy:parab}
Let $\phi$ be a function satisfying Definition \ref{gauge} and let $w\in H^2$. Then there exists a constant $C(k_G)$, such that
   \begin{align*}
 I_{w}:= -( \phi^{-2}\dx^4w, w) \le -\frac{1}{C(k_G)}\norm{w}_{H^2}^2+  C(\tld M(t)) \norm{w}_{L^{2}}^2
 \end{align*}
\end{lem}
\begin{proof}
Integrating by parts twice gives
 \begin{align*}
  & I_w = -\left[(\phi^{-2}\dx^2 w,\dx^2 w)  + (\dx^2(\phi^{-2})\dx^2 w,w) -2 (\dx^2(\phi^{-2}) \dx w,\dx w)\right],
 \end{align*}
 where we have used $\dx^2 w\cdot \dx w = \frac{1}{2}\dx[\dx w ^2]$ for one more integration by parts in the last term. Using Cauchy-Schwarz implies
 \begin{align}\label{eps-energy}
   I_w \le -(\phi^{-2}\dx^2 w,\dx^2 w) + \norm{\phi^{-2}}_{W^{2,\infty}}(\norm{w}_{H^1}^2 + \norm{w}_{H^2}\norm{w}_{L^2})].
 \end{align}
 Interpolating $\norm{w}_{H^1}^2\le\norm{w}_{H^2}\norm{w}_{L^2}$ and using \eqref{guage-der} gives
 \begin{align*}
   I_w \le -(\phi^{-2}\dx^2 w,\dx^2 w) + C(\tld M(t)) \norm{w}_{H^2}\norm{w}_{L^2}.
 \end{align*}
 A Cauchy-Schwarz estimate gives, for $\alpha>0,$
 \begin{align*}
   I_w \le -(\phi^{-2}\dx^2 w,\dx^2 w) + \alpha \norm{w}_{H^2}^2 + \alpha^{-1} C(\tld M(t))\norm{w}_{L^2}^2.
 \end{align*}
 Using the upper bound for $\phi$ from \eqref{gauge-below} we estimate
$$(\phi^{-2}\dx^2 w,\dx^2 w)\ge \frac{1}{C(k_G(t))} \norm{w}_{\dot H^2}^2.$$
Making the choice $\alpha = \frac{1}{2C(k_G(t))}$ completes the proof.\end{proof}

\begin{proof}[Proof of Theorem \ref{thm:lin}]
We first claim that demonstrating \eqref{lin:ref} reduces to showing
\begin{equation}\label{dt:lin}
  \dt \norm{v}_{L^{2}}^2 \le C(\tld M(t))(\norm{v}_{L^{2}}^2 + \norm{v}_{L^{2}}\norm{h}_{L^{2}})
  -\frac{\eps}{ C(k_G)}\norm{w}_{H^2}^2.
\end{equation}
Indeed, if \eqref{dt:lin} were true, we ignore the $H^2$ term and use Grownwall's lemma and Cauchy-Schwarz to get
\begin{align*}
  \sup_{0\le t'\le t}\norm{v(t')}^2_{L^2}\le e^{\int_0^t C(\tld M(t'))dt'}\left[\norm{v_0}+\int_0^t \norm{h(t')}_{L^2}\right]^2.
\end{align*}
Applying \eqref{v} demonstrates \eqref{lin:ref} except for the $\norm{w}_{H^2}$ term on the left. To get it, we rearrange \eqref{dt:lin}:
\begin{align*}
  \frac{\eps}{ C(k_G)}\norm{w}_{H^2}^2 \le C(\tld M(t))(\norm{v}_{L^{2}}^2 + \norm{v}_{L^{2}}\norm{h}_{L^{2}}) -  \dt \norm{v}_{L^{2}}^2.
\end{align*}
Integrating in time and using \eqref{lin:ref} (without the $H^2$ term) we get:
\begin{align*}
  & \eps\norm{w}_{H^2}^2 \le C(k_G)\left[\int_0^t C(\tld M(t'))(\norm{v}_{L^{2}}^2 + \norm{v}_{L^{2}}\norm{h}_{L^{2}})(t') dt' -  \norm{v(t)}_{L^{2}}^2 +  \norm{v_0}_{L^{2}}^2\right]\\
  & \le C(k_G(t))\exp\left(\int_0^t C(\tld M(t')) dt'\right)\left( \norm{w_0}_{L^2} + \norm{h}_{L^1_t L^2}\right).
\end{align*}

It remains to establish the estimate \eqref{dt:lin}. To do so we return to \eqref{v:energy} with $\phi$ from Lemma \ref{gauCh}. Using Lemma \ref{enrg} implies
 \begin{equation*}
   \dt \int \abs{v}^2 = - 2 (c_0 v,v) +(h,\phi^{-1}v)- ( \eps \dx^4(\phi v), \phi^{-1}v)=:I + II +III,
 \end{equation*}
 where $c_0$ is defined by Lemma \ref{enrg} applied to $L=L_\phi$. In particular,
 \begin{equation*}
   \norm{c_0}_{L^\infty_{x,t}}\le C(\tld M(t)),
 \end{equation*}
 which gives an estimate of $$I\le C(\tld M(t))\norm{v}_{L^{2}}^2.$$
 Next we use \eqref{gauge-below} to estimate the term $II:$
 $$II=|(h,\phi^{-1}v)|\le C(k_G(t))\norm{v}_{L^{2}}\norm{h}_{L^{2}}.$$
  Finally, note that $III = \eps I_w$ for $I_w$ from Lemma \ref{energy:parab} and $w$ from \eqref{gauged:var}. Hence from Lemma \ref{energy:parab} and \eqref{v}:
   $$
   III \le \eps C(\tld M(t)) \norm{v}_{L^{2}}^2- \frac{\eps}{ C(k_G)}\norm{w}_{H^2}^2.
   $$
   Combining the estimates for $I,$ $II,$ and $III$ establishes \eqref{dt:lin} and concludes the proof.
 \end{proof}

\section{Nonlinear {\it a priori} estimates}
 To construct solutions of \eqref{generalEquation}, we begin with the solutions of the parabolically regularized equation \eqref{parab:non}. The goal of this section is to establish a uniform {\it{a priori}} estimate on the dispersion and on the high norms of the solutions of \eqref{parab:non}. Namely, we show that the solution cannot grow too fast for a certain time, with this time depending on the size of the initial data, dispersion and a balance of parameters $\eps$ and $\delta$. Our main nonlinear estimate is summarized in the following propositions, proofs of which will occupy most of this section.

 \subsection{Main propositions}
 \begin{prop}\label{unif:est1}
 Let $T'>0$ be given. There exists increasing functions $C_1(\cdot,\cdot)$ and $C_3$ with $C_1\ge 1$ and $C_3\ge 1$
 such that
 the following inequalities hold. Let $\ue \in C^1_{T'} H^{8}\cap C^0_{T'} H^{12}$ be a solution of \eqref{parab:non}. Let $M_\eps(t)$, $k(t)$ be as in \eqref{high-norm} and \eqref{dispersion+low-norm}, respectively. Then for $t\le T',$
   \begin{align}
     k(t)& \le k(0) \frac{1}{1-tC_1(M_\eps(t),k(0))} + tC_3(M_\eps(t)).\label{low-norm-est}
   \end{align}
    \end{prop}
  \begin{prop}\label{uniform:est}
For $T'>0$ as before, there exist functions $C_2$ and $C_4$ both increasing and bounded below by $1$, so that for $\ue \in C^1_{T'} H^{8}\cap C^0_{T'} H^{12}$ as in Proposition \ref{unif:est1},%
    \begin{align}\label{uniform:n}
 M_\eps(t)& \le C_2(k(t),tC_4(M_\eps(t)))\left[ M_\eps(0)+t C_4(M_\eps(t))\right].
  \end{align}
\end{prop}
\begin{rem}
  Note that we demand that all the functions $C_{1},$ $C_{2},$ $C_{3},$ and $C_4$ are only dependent on $\eps\ge 0$ through $M_\eps$ and not in any other way.
  \end{rem}
 Before proving these propositions, we discuss their implications. Essentially, we want to obtain a lower bound on the lifespan of the solution $\ue,$ independent of the values of $\eps$ and $\delta$. One way to achieve this is to find $M>M_\eps(0)$, so that whenever the solution norm is trapped in the region $M_\eps(t)\in (M,2M)$, a ``substantial'' amount of time must have passed. When combined with the existence result for a parabolic regularization, Proposition \ref{parab:prop}, these propositions allow us to create solutions to \eqref{generalEquation}, whose $H^7$-norm and time of existence are independent of the regularizations. We provide the details of this informal outline before giving the proof of the Propositions.\\ %

 We also want to discuss an interesting technical challenge here. The linear estimate \eqref{lin:ref} applied to \eqref{parab:non} may na\"{\i}vely suggest that it may be possible to prove an energy estimate
  $$
  \norm{\ue}_{H^s} \le O(\norm{u_{0,\eps}}_{H^s}),
  $$
  and of course such estimates are valid for KdV and other semilinear equations. However, the validity of such an estimate, even on the linear level, relies on the non-vanishing of dispersion and finiteness of ``antidiffusion" (as captured in coefficient $k_G(t)$ in \eqref{linear-coeff} and $k(t)$ in the nonlinear problem \eqref{parab:non}). For this reason, our argument requires the combination of Propositions \ref{unif:est1} and \ref{uniform:est}. In some form Proposition \ref{uniform:est} estimates the $H^7$ Sobolev norm for the regularized problem and is based on the refined linear estimate (Theorem \ref{thm:lin}), while Proposition \ref{unif:est1} gives an estimate of dispersion and its proof is cruder, but just as essential.
  \begin{coro}\label{apriori}
    Define \begin{align}\label{M}
      M=2C_2(4k(0), 4k(0))[ M_\eps(0)+2k(0)], \text{ for the function }C_2 \text{ from Proposition }\ref{uniform:est}.
    \end{align} Then there exists $T=T(M,k(0))$, so that if  $t>0$ and
  \begin{align}\label{Mk}
    M < M_\eps(t),%
  \end{align}
    then $t> T$.
  \end{coro}
  \begin{proof}
    Let $T =\min \left(\frac{2k(0)}{C_3(2M)}, \frac{1}{ 2C_1(2 M,k(0))},\frac{2k(0)}{C_4(2M)}\right)$ for $C_{1},$ $C_3$ and
    $C_{4}$ from \eqref{low-norm-est} and \eqref{uniform:n}. Without loss of generality, since the solution is continuous in time, we may assume more than \eqref{Mk} holds:
      \begin{equation}\nonumber
      M<M_{\eps}(t)\leq 2M.
      \end{equation}
Now, assume for the sake of contradiction that $t\le T$. From \eqref{low-norm-est} for $t\le T$ we obtain:
    \begin{align}\label{kt-k0}
      k(t)\le k(0) \frac{1}{1-TC_1(M_\eps(t),k(0))} + T C_3(M_\eps(t))      \le 4k(0).
    \end{align}
    With this bound on $k(t)$ we apply \eqref{uniform:n} and use $t\le T$:
    \begin{align*}
      M_\eps(t) \le C_2(4k(0), 2k(0))[M_\eps(0) + 2k(0)] < M.
    \end{align*}
     The last estimate contradicts \eqref{Mk}, thus the only way to avoid the contradiction is to conclude that $t> T$.
  \end{proof}
  Note that the estimate \eqref{kt-k0} in Corollary \ref{apriori} only depends on $T$ and the size of $M_\eps(t)$. In particular, we can conclude the following.
    \begin{rem}\label{rem:k}
    With the choice of $M$ and $T$ from the prior Corollary \ref{apriori}, we get that if $M_\eps(t)\le 2M$ and $0\le t\le T$, then
    \begin{align*}
      k(t)\le 4k(0).
    \end{align*}
  \end{rem}
\begin{coro}\label{2M} Let $T>0$ be as in Corollary \ref{apriori}. Suppose $\ue\in C_{T'}H^{12}$ is a solution of \eqref{parab:non} for $T'\le T$. Then
\begin{align}
  M_\eps(T')\le 2M.
\end{align}
\end{coro}
 \begin{proof}
   By continuity of $\norm{\ue(t)}_{H^8}$, $M_\eps(t)$ is a continuous function. Let $0\le t^*\le T'$ be the smallest value so that $\norm{\ue(t^*)}_{H^8}\ge 2M$ (if such $t^*$ exists). Recalling that the function $C_2$ from Proposition \ref{uniform:est} satisfies the bound $C_2\ge 1$, we deduce from the definition of $M$ in Corollary \ref{apriori} that
   \begin{align*}
     M\ge 2 C_2 M_\eps(0) > M_\eps(0).
   \end{align*}
   We conclude that $t^*>0$ and $\norm{\ue(t^*)}_{H^8}=2M$. Hence by Corollary \ref{apriori} applied to $\ue$,
   we conclude $t^*\ge T$. This is a contradiction, as $t^*\le T'\le T$ by the hypothesis.
   The only alternative is $M_\eps(t)\le 2M$ for all $t\in [0,T']$.
 \end{proof}
 \begin{coro}\label{uniform:exist}
 Let $\delta_0$ be as in Remark \ref{M0} and $0<\eps \le \delta^2\le \delta_0^2$.
 Then there exists a $T=T\left(\frac{1}{\norm{u_0}_{H^7}},\norm{u_0}_{H^7},\lambda(0)\right)>0$,
 such that the maximal interval for wellposedness of \eqref{parab:non}, as stated in Proposition \ref{parab:prop}, contains the interval $[0,T]$ for all $\eps$ specified above. Furthermore, $M_\eps(T)\le 2M$ for $M$ from \eqref{M}.
 \end{coro}
 \begin{proof}
   Note that by the Proposition \ref{parab:prop}, the proof reduced to an estimate of $\norm{\ue}_{H^4}$ or higher. We set $s=12$ in the Proposition \ref{parab:prop}, so that Corollary \ref{2M} applies (and by Lemma \ref{lem:regularize}, $(u_{0})_{\delta}\in H^{12}$ for $\delta>0$). \\

   We combine Remark \ref{M0} with Corollary \ref{apriori}, to conclude that the bound $M$ and a time interval $T>0$ from that Corollary are increasing in the following parameters:
   \begin{align*}
     M=M(\norm{u_0}_{H^7},\lambda(0)),\\
     T=T\left(\frac{1}{\norm{u_0}_{H^7}},\lambda(0)\right),
   \end{align*}
   and independent of $\eps$. Note, that we have used \eqref{dispersion+low-norm} to relate $k(0)$ with $\lambda$.\\

    We now conclude using Corollary \ref{2M}, that for all $T'\le T,$
    \begin{align*}
      \norm{\ue}_{C_{T'} H^4}\le M_\eps(T')\le 2M.
    \end{align*}
    Hence up to time $T$, the $H^4$ norm cannot blow up and \eqref{parab:non} must be wellposed.
 \end{proof}
  The rest of the section is organized as follows. Proposition \ref{unif:est1} can be proved directly from \eqref{parab:non} and we prove it first. The estimate \eqref{uniform:n} is more involved and is done in Sections \ref{sec:reduction} through \ref{subsec:uniform-proof}.
 \subsection{Proof of Proposition \ref{unif:est1}}
 The content of Proposition \ref{unif:est1} is a bound for the dispersion and for a low norm of the solution.
  To estimate the dispersion and the $H^4$-norm, we first estimate the time derivative of the solution via the evolution equation in \eqref{parab:non}.
   \begin{prop}\label{lem:dt}
  Suppose $\ue \in C^1_t H^4\cap C^0_t H^8$ satisfies \eqref{parab:non}. Then
  \begin{align}
\norm{\dt \ue}_{H^4}\le C(\norm{\ue}_{H^7},\eps\norm{\ue}_{H^8}).
\end{align}
\end{prop}
We omit the proof as it is fairly obvious, since the evolution equation involves at most three derivatives inside the function
$f,$ and the parabolic term has an $\varepsilon$ and a fourth derivative.

We next set down some slightly unconventional notation in order to simplify chain rule computations.
\begin{rem}\label{rem:z}
  Denote $z_{-1}=x$ and $z_{-2}=t$. so that $(\vz, z_{-1},z_{-2})=(\dx^3 \vue,x,t)$. This way $$\dx \vz = (\dx[\dx^3\vu^\eps],1,0)$$
\end{rem}
\begin{proof}[Proof of Proposition \ref{unif:est1}]
  From \eqref{dispersion+low-norm}, we need to analyze $\norm{u_{\varepsilon,\delta}(t)}_{H^4}$ and $\frac{1}{\lambda(t')}$, which we do separately.

  From the Fundamental Theorem of Calculus
 \begin{align*}
   u_{\varepsilon,\delta}(t) = (u_{0})_{\delta} + \int_0^t \dt u_{\varepsilon,\delta}(t') dt'.
 \end{align*}
 Taking $H^4$-norms implies \begin{align*}
   \norm{\ue(t)}_{H^4}\le \norm{(u_{0})_{\delta}}_{H^4}+ \int_0^t\norm{\dt u_{\varepsilon,\delta}(t')}_{H^4} dt'.
 \end{align*}
Hence from Proposition \ref{lem:dt},
\begin{align}\label{dnorm-est1}
  \norm{\ue(t)}_{H^4} \le \norm{(u_{0})_{\delta}}_{H^4} + C(M_\eps(t))t.
\end{align}

We now proceed to an estimate of $\lambda(u_{\varepsilon,\delta}(t))$, which is deduced by a similar method.
First, from the definition \eqref{Dispersion} and the Fundamental Theorem of Calculus,
\begin{align*}
  f_{z_3} (\dx^3 \vue(x,t)) = f_{z_3} (\dx^3 (\vu_{0})_{\delta}(x))+\int_0^t \dt [f_{z_3} (\dx^3 \vue)(x,t')] dt'.
\end{align*}
Expanding the time derivative above yields
\begin{align*}
  f_{z_3} (\dx^3 \vue)\ge  f_{z_3} (\dx^3 (\vu_{0})_{\delta}) - t\sum_{j=-2}^3 \norm{f_{z_3,z_j}(\dx^3\vue)\dt [z_j(x,t)]}_{L^\infty_{t,x}},
\end{align*}
where we have used the notation from Remark \ref{rem:z}. Thus
\begin{align*}
  f_{z_3} (\dx^3 \vue)\ge  f_{z_3} (\dx^3 (\vu_{0})_{\delta}) - t C(\norm{\ue}_{L^\infty_{t}W^{3,\infty}_x})(1+\norm{\dt \ue}_{L^\infty_{t}W^{3,\infty}_x}).
\end{align*}
Hence by Sobolev embedding we obtain
\begin{align*}
  \frac{1}{f_{z_3} (\dx^3 \vue)} \le \frac{1}{f_{z_3} (\dx^3 (\vu_{0})_{\delta})\left[1-T C(\norm{\ue}_{W^{1,\infty}_t H^4_x},\frac{1}{f_{z_3} (\dx^3 (\vu_{0})_{\delta})})\right]}.
\end{align*}
Using Proposition \ref{lem:dt} to estimate $\dt u_{\varepsilon,\delta},$ and using \eqref{Dispersion}, allows us to express the above inequality as
$$
\lambda(t) \le \lambda(0) \frac{1}{1-T C(M_\eps(t),k(0))}.
$$
Combining this result with \eqref{dnorm-est1} concludes the proof.
\end{proof}

\subsection{Reduction to the linear estimate}\label{sec:reduction}

To prove Proposition \ref{uniform:est}, we aim to use linear estimate \eqref{lin:ref} from Theorem \ref{thm:lin} for the equation that $\dt \dx^n \ue$ satisfies for $n\ge 3$. We begin with differentiation of \eqref{parab:non}.

\begin{rem}\label{rem:fx-vs-dxf}
  Note that we distinguish between $f_x(\dx^3\vue) = [\dx f](\vz,x,t)\mid_{\vz = (\dx^3 \ue,\ldots \ue)}$ and
  \begin{align*}
    \dx [ f(\dx^3\vue )] = \sum_{i=0}^3 f_{z_i}(\dx^3\vue) \dx^{i+1} \ue + f_x(\dx^3\vue) = \sum_{i=-1}^3 f_{z_i}(\dx^3\vue),
  \end{align*}
   where we use the notation from Remark \ref{rem:z} (i.e. $z_{-1}=x$).

   Occasionally, to be more efficient we may omit the variable $z$ and denote derivatives of $f$ with just the indices, i.e. $f_{-2}=f_t$ and $f_{-1,2}=\frac{\partial^2 f}{\partial x \partial {z_2}},$ for example.
\end{rem}

\begin{prop}\label{prop:dt3}
  Suppose $\ue$ solves \eqref{parab:non}. Then \begin{align}\label{nonlinear-dx3}
  & \dt \dx^3 \ue + f_{z_3}\dx^6 \ue + \left( f_{z_2}+3\dx [ f_{z_3}(\dx^3\vue)]\right)\dx^5 \ue + \tld f(\dx^4\vue) =-\eps\dx^7 \ue,
\end{align}
where %
$$\tld f\in C^1_t C^{7}_z W^{7,\infty}_x.$$
\end{prop}
\begin{rem}
For completeness, we present the structure of $\tld f$ in the proof.
\end{rem}

\begin{proof}
  Differentiating $f$ once we obtain
  \begin{align*}
    \dx [f] = \sum_{i=-1}^3 f_i \dx z_i.
  \end{align*}
 In the next step we have a product rule in addition to the chain rule:
  \begin{align*}
    \dx^2 [f] = \sum_{i,j=-1}^3 f_{i,j}\dx z_i \dx z_j + \sum_{i=0}^3 f_i \dx^2 z_i.
  \end{align*}
  Note that we eliminated terms that are obviously equal to zero, such as $\dx^2 z_{-1}=\dx^2 (x)=0$.
We then claim that applying a third derivative produces the following expression:
  \begin{equation}\label{davidNumberedThis}
    \dx^3[f]  = \sum_{i=0}^3 f_i \dx^3 z_i + \sum_{i,j\le 3} f_{i,j} \dx^2 z_i \dx z_j\\
     + 2 \sum_{i,j\le 3} f_{i,j}\dx^2 z_i \dx z_j + \sum_{i,j,k\le 3} f_{i,j,k} \dx z_i\dx z_j\dx z_k.
  \end{equation}
  The first two terms on the right-hand side of \eqref{davidNumberedThis} include derivatives of the term $ \sum_{i=0}^3 f_i \dx^2 z_i$, while in the third term on the right-hand side we have introduced a coefficient $2$ obtained from changing the index of terms with $\dx^2 z_j \dx z_i$ from $i$ to $j$. Further note that the second and third terms on the right-hand
 side may be combined.

  From now on, most terms end up in the $\tld f$ function that we will define below. We make a couple of observations. First, the only term with $6$ derivatives on $u_{\varepsilon,\delta}$ is $f_3 \dx^3 z_3$.\\

  Second, we inspect the terms with $\dx^5 u_{\varepsilon,\delta}.$ They are
  \begin{align*}
    f_2 \dx^3 z_2 + 3\sum_{j=-1}^3 f_{3,j} \dx^2 z_3 \dx z_j =f_2 \dx^3 z_2 + 3\dx[ f_3(\vec{z})].
  \end{align*}

  Finally, all the remaining terms combine into the term we call $\tld f$, written explicitly as follows:
  \begin{align}\label{tldf}
  \tld f = \sum_{i,j,k=-1}^3 f_{i,j,k} \dx z_i \dx z_j \dx z_k + 3 \sum_{i<3,j\le 3} f_{i,j} \dx^2 z_i \dx z_ j + \sum_{i=0}^1 f_i \dx^3 z_i. %
\end{align}
  \end{proof}
From now on the coefficient of the second highest space derivative $\dx^5 \ue$ above will change in the ``linear fashion'' (i.e. similar in structure to differentiation of \eqref{lin}). We continue with differentiation until this pattern appears for lower order terms by considering $n\ge 7$:
\begin{lem}\label{lem:deriv}
  Let $\ue$ solve \eqref{parab:non}. Then $\dx^n\ue$ satisfies the following expression for $n\ge 7:$
  \begin{align}\label{differentiate-nonlinear-0}
\begin{split}
  & \dt \dx^n \ue + f_{z_3}(\dx^3\vue)\dx^{n+3} \ue + \left( f_{z_2}(\dx^3\vue)+n\dx [ f_{z_3}(\dx^3\vue)]\right)\dx^{n+2} \ue\\
&  + a_{1,n}(\dx^5\vue)\dx^{n+1} \ue + a_{0,n}(\dx^6\vue)\dx^{n} \ue=-\eps\dx^{n+4} \ue+\tld f_n(\dx^{n-1}\vue),
\end{split}
\end{align}
where the  functions $a_{0,n}$, $a_{1,n}$ and $\tld f_n$ are in $C^1_t C^{11-n}_z W^{11-n,\infty}_x$.
\end{lem}
\begin{rem}
  Note that the number $11$ in the index $11-n$ in the regularity above comes from the condition (A1) in \eqref{coeff-bdd}. Therefore a smoother nonlinear function $f$ would allow $n>11$.
\end{rem}
Essentially, Lemma \ref{lem:deriv} is a tedious application of the chain rule, a.k.a. Faa Di Bruno formula. We do this in full detail, but we first make some reductions.\\

Once the equation \eqref{differentiate-nonlinear-0} is known to be valid for some $n$, say $n=7$, 
we may differentiate for higher $n$ inductively getting the following formulas:
\begin{align}\label{higher-coeff}
\begin{split}
  a_{1,n+1}(\dx^5\vue)& =a_{1,n}(\dx^5\vue)+\dx\left[ f_{z_2} + n \dx[f_{z_3}(\dx^3\vue)]\right],\\
  a_{0,n+1}(\dx^6\vue)& =a_{0,n}(\dx^6\vue)+\dx[a_{1,n}(\dx^5\vue],
\end{split}\\
    \tld f_{n+1}(\dx^{n}\vue)& =\dx[\tld f_{n}(\dx^{n-1}\vue)] + \dx[a_{0,n}(\dx^6\vue)]\dx^n\ue \label{f-n}.
  \end{align}
However, to establish the structure of \eqref{differentiate-nonlinear-0} initially for $n=7$ requires doing a 
full expansion of terms of order $n$ similar to \eqref{davidNumberedThis}. Before doing this we make some
calculations to prepare.

For the terms of order $n+3$ and the terms of order $n+2,$ 
the structure in \eqref{differentiate-nonlinear-0} appears as early as $n=3$, which is why we 
have established Proposition \ref{prop:dt3} first. We thus establish such a lemma first by differentiating \eqref{nonlinear-dx3} $n-3$ times and focusing on higher order terms.
\begin{lem}\label{lem:deriv:mild}
Let $\ue$ solve \eqref{parab:non}. Then $\dx^n\ue$ satisfies the following expression for $n\ge 3:$
  \begin{align}\label{differentiate-nonlinear-top}
\begin{split}
  & \dt \dx^n \ue + f_{z_3}(\dx^3\vue)\dx^{n+3} \ue + \left( f_{z_2}(\dx^3\vue)+n\dx [ f_{z_3}(\dx^3\vue)]\right)\dx^{n+2} \ue\\
&  + a_{1,n}(\dx^{n+1}\vue)\dx^{n+1} \ue + a_{0,n}(\dx^{n}\vue)\dx^{n} \ue=-\eps\dx^{n+4} \ue+\tld f_n(\dx^{n-1}\vue)
\end{split}
\end{align}
where the  functions $a_{0,n}$, $a_{1,n}$ and $\tld f_n$ are in $C^1_t C^{11-n}_z W^{11-n,\infty}_x$.
\end{lem}
Note that the main difference between this lemma and Lemma \ref{lem:deriv} is the structure of the coefficients $a_{1,n}$ and $a_{0,n}$. In this lemma we simply collect all terms of order $n+1$, $n$ and below into $a_{1,n}\dx^{n+1}\ue$, $a_{0,n}\dx^{n}\ue$ and $\tld f_n$, respectively, without excluding impossible terms.
\begin{proof}
Differentiate the equation \eqref{davidNumberedThis} $n-3$ times aiming to get a result similar to \eqref{nonlinear-dx3}. First, the only way to obtain the term $\dx^{n+3}\ue$
when differentiating \eqref{nonlinear-dx3} $\dx^{n-3}$ times is to keep the $f_{z_3}$ term and place all derivatives on $\dx^6\ue$.

  Second, for the $\dx^{n+2}\ue$ term, a part of the coefficient comes from differentiating the
  $f_{z_2}+ 3 \dx[f_{z_3}(\vue)]\cdot\dx^5\vue$ term and the remaining
  $(n-3)\dx[f_{z_3}(\vue)]$ comes from differentiating the $f_{z_3}\dx^6\ue$ term.

Third, note that all the coefficients in the equation \eqref{differentiate-nonlinear-0} are obtained from $n$ derivatives placed on the nonlinear function $f$ from \eqref{generalEquation}. By the condition $(A1)$ we have assumed $f(\vz,x,t) \in C^1_t C^{11}_{\vz}W^{11,\infty}_x$ and hence $\partial^\alpha_z\partial^\beta_x f(\vz,x,t) \in C^1_t C^{11-n}_{\vz}W^{11-n,\infty}_x$ for $|\alpha|+\beta=n$.

Finally, we arrange all terms of order $n+1$ in the form $a_{1,n}\dx^{n+1}\ue$, terms of order $n$ as $a_{0,n}\dx^{n}\ue$ and terms of order less than $n$ as $\tld f_n$.
\end{proof}
We now return to the explicit analysis of $a_{1,n}$, $a_{0,n}$ and $\tld f_n$.
In order to track the dependence of the coefficients of order $7$ and higher, we need an expansion of all the possible terms that arise in $\dx^n[f]$ similar to \eqref{davidNumberedThis} that works for all $n\ge 1$.

 Since the expansion is messier, we will change the notation and explain this change in the case of $n=3$. In order to keep track of terms of the type $\dx^\alpha\dz^\beta f$, we organize them by the index $k=|\alpha|+|\beta|$. The terms containing $\dx^\alpha\dz^\beta f$ will be paired with a polynomial of degree $k$ in derivatives of $u$, which we need to label. For each of the $\beta$ derivatives, which we enumerate with a label $l$, there are $4$ choices that each $z$-derivative produces: $z_{-1}=x$, $z_0=u,$ \ldots, $z_3=\dx^3u$ (by the notation of the Remarks \ref{rem:z} and \ref{rem:fx-vs-dxf}). We reserve the subscript $j$ for each of those choices, i.e. $j=0$ will lead to a $u$ term, rather than a $\dx u$ term from $j=1$.

 Finally, if we include $z_{-1}=x$ each of the $n$ derivatives hits a $z$ term. 
 For example a term $$f_{1,0,-1}(\vz)\dx^2 z_{1} \dx z_0 \dx z_{-1}$$ 
 is obtained by differentiating the nonlinear function $f$ 
 three times in direction $z_1$, $z_0$ and $z_{-1}=x$, followed by one more derivative of the function $z_1=\dx u$ to get to $\dx^2 z_1$. We label the index $\dx^2 z_2$ as $i_1=2$, the index of $z_0$ as $i_2=1$ and the index of $z_{-1}$ as $i_3=1$ proceeding from highest index to lowest. With this language in mind, $\dx^n f$ can be expressed as:
\begin{align}\label{diff:n}
  \dx^n[f]=\sum_{k=1}^n
  \sum_{\substack{ i_1\ge i_2\ge\ldots \\
            \ge i_k\ge 1;\\
            \sum_{m=1}^k i_l=n}}
            \sum_{l=1}^k\sum_{\substack{-1\le j_l\le 3} }
    C_{\vec{i},\vec{j}}  f_{j_1,\ldots j_k}(z)\dx^{i_1}z_{j_1}\cdots \dx^{i_k} z_{j_k}.%
\end{align}
Note that the $C_{i,j}$ parameters appear from the rearrangement of terms and product rule. %

We demonstrate the change of notation from the case of $n=3$ in \eqref{davidNumberedThis}. The terms for $k=1$ are $\sum_{i=0}^3 f_i \dx^3 z_i$, which we relabel as
\begin{align*}
  \sum_{i=0}^3 f_i \dx^3 z_i = \sum_{j_1=-1}^3 f_{j_1} \dx^3 z_{j_1}.
\end{align*}
The terms for $k=2$ are relabeled as
\begin{align*}
  3 \sum_{i,j\le 3} f_{i,j}\dx^2 z_i \dx z_j = \sum_{j_1=-1}^3\sum_{j_2=-1}^3 3 f_{j_1,j_2} \dx^2 z_{j_1}\dx z_{j_2}.
\end{align*}
and the terms for $k=3$ become
\begin{align*}
  \sum_{i,j,k\le 3} f_{i,j,k} \dx z_i\dx z_j\dx z_k = \sum_{j_1,j_2,j_3=-1}^3 f_{j_1,j_2,j_3} \dx z_{j_1}\dx z_{j_2}\dx z_{j_3}.
\end{align*}
\begin{proof}[Proof of Lemma \ref{lem:deriv}]
By  Lemma \ref{lem:deriv:mild} and the inductive formula \eqref{higher-coeff} the proof reduces to the analysis of 
the coefficients $a_{1,7}$ and $a_{0,7}$.\\

We now return to use the notation in \eqref{diff:n} to confirm that $a_{1,7}$ depends on no more than 
five derivatives of $\ue$. It is easier to do the analysis for $k=1$, $k=2$ and greater separately:
  \begin{align*}
    \dx^7[f] & = \sum_{j_1=0}^3 f_{j_1} \dx^7 z_{j_1} + \sum_{i_1=4}^6\sum_{j_1,j_2\le3} C_{\vec{i},\vec{j}} f_{j_1,j_2} \dx^{i_1} z_{j_1}\dx^{7-i_1} z_{j_2}\\ &
    +\sum_{k=3}^7\sum_{\substack{i_1\ge i_2\ge\ldots \\
  \ge i_k\ge 1;\\ \sum_{m=1}^k i_l=7}}\sum_{l=1}^k\sum_{-1\le j_l\le 3}  C_{\vec{i},\vec{j}}f_{j_1,\ldots j_k}(z)\dx^{i_1}z_{j_1}\cdots \dx^{i_k} z_{j_k}\mid_{z=(\dx^3\vue,x,t)}.
  \end{align*}
  In the sum above we have excluded some of the vanishing terms, like $j_1=-1$ for $k=1$. We also reorganized the sum for $k=2$, relabeling $i_2=7-i_1$ and exploiting $7\ge i_1\ge i_2$.

  The terms for $i_1+j_1=10$, as well $i_1+j_1=9$, are already accounted for in \eqref{differentiate-nonlinear-0} as coefficients of $\dx^{n+3}\ue$ and $\dx^{n+2}\ue$ for $n=7$. Therefore, to obtain $a_{7,1}$, the coefficient for $\dx^8 \ue$, we need to focus on terms with $i_1+j_1=8$. Note that as $i_2\le 7-i_1\le 3$ if $k\ge 2$, it is impossible to obtain $i_2+j_2=8$ for $n=7$.

  We do an enumeration of all terms of order $8$ using $i_1+j_1=8$. If $i_1=7$, then $k=1$ to satisfy $i_1+i_2+\ldots=7$ and all $i$ terms being positive. If $i_1=6$, $j_1=2$ then $k=2$ with $i_2=2$. If $i_1=5$, $j_1=8-i_1=3$ then either $k=3$ and $i_2=i_3=1$ or $k=2$ and $i_2=2$. Explicitly, all terms of order $8$ can be listed as follows:
  \begin{align}
    f_{1}\dx^7 z_{1} + \sum_{\substack{i_1=5}}^6\sum_{j_2=-1}^3 C_{\vec{i},\vec{j}} f_{j_1=8-i_1,j_2} \dx^{i_1} z_{j_1}\dx^{7-i_1} z_{j_2}
     + \sum_{j_2,j_3=-1}^3 C_{\vec{i},\vec{j}}f_{z_3,z_{j_2},z_{j_3}}(\vz)\dx^{5}z_{3}\dx z_{j_2} \dx z_{j_3}.
  \end{align}
We now substitute $\vz=(\dx^3 \ue,\ldots \ue,x,t)$ to get:
\begin{align*}
  \dx^8\ue\cdot a_{1,7}:=\dx^8\ue\cdot [f_1+\sum_{j_2=-1}^3 C_{\vec{i},\vec{j}} f_{2,j_2}\dx^{j_2+1}\ue +\sum_{j_2=0}^3 C_{\vec{i},\vec{j}}f_{3,j_2}\dx^{j_2+2}\ue\\ + \sum_{j_2,j_3=-1}^3 C_{\vec{i},\vec{j}} f_{2,j_2,j_3}\dx^{j_2+1}\ue\dx^{j_2+1}\ue ].
\end{align*}
We now claim that $a_{1,7}=a_{1,7}(\dx^5\vue)$. Indeed, the
nonlinear function $f$ depends on at most $\dx^3 \ue$, and the highest derivative of $\ue$ possible inside is the term $\dx^{j_2+2}\ue$, where $j_2 =3$ is possible.

The analysis of the $\dx^7\ue$ coefficient, $a_{0,7}$ is similar.
\end{proof}

\begin{coro}\label{lem:remainder11-prep}
  By restricting inadmissible indices in the expansion of $\dx^n [f]$ in \eqref{diff:n} we can also get an explicit description of $\tld f_n:$ 
\begin{align*}
  \tld f_n& =
  \sum_{k=2}^{n}\sum_{\vec{i},\vec{j}\in \mathcal{S}_k^n}
    C_{\vec{i},\vec{j}}\cdot f_{j_1,\ldots j_k}(z)\dx^{i_1}z_{j_1}\cdots \dx^{i_k} z_{j_k},
  \end{align*}
  where the $k-$tuples $\vec{i}$, $\vec{j}$ of admissible indices $\mathcal{S}_k^n$ are defined by
  \begin{align*}
    \mathcal{S}_k^n =\left\{1\le l \le k; i_1\ge i_2\ge\ldots\ge i_k\ge 1;  \sum_{l=1}^k i_l=n; -1\le j_l\le 3; i_l+j_l<n;\max\{j_l,1-i_l\}\ge 0 \right\}.
  \end{align*}
\end{coro}
Note that we will use this description for $n=11$ later in the paper.
\begin{proof}
  From \eqref{diff:n} we remove all terms involving more than $n-1$ derivatives of $\ue$. In the language of \eqref{diff:n} this means $i_l+j_l<n$. The indices in $\mathcal{S}_k^n$ are those that remain organized by the parameter $k$ that counts the number of $z$ derivatives on the nonlinear function $f$.
  
  Note that $k=1$, for example, leads to $i_1=n$ and hence only contain terms of order higher than $n-1$. 
  Similarly the condition $\max\{j_l,1-i_l\}\ge 0$ is just a statement that more than two derivatives annihilate $z_{-1}=x$.
\end{proof}

We aim to consider a linear problem \eqref{lin} with coefficients from \eqref{differentiate-nonlinear-0}:
\begin{align}\label{nonlinear-linear}
\begin{split}
  \begin{cases}
   \dt w + \sum_{j=0}^3 a_j(x,t)\dx^j w = -\eps\dx^4 w+\tld f_n(\dx^{n-1}\vue),\\
  w(x,0)= \partial_{x}^{n}(u_{0})_{\delta}(x),
\end{cases}\\
\text{where } a_3: =  f_{z_3}(\dx^3 \vue), \hspace{ 10 pt}
a_2: = \left( f_{z_2}+n\dx [ f_{z_3}(\dx^3\vue)]\right), \text{ and } a_j = a_{j,n}(\dx^{6-j}\vue).
\end{split}
\end{align}
Note that $w=\dx^{n}\ue$ is a solution of that linear equation and we can apply Theorem \ref{thm:lin} to it to find the
following estimate:
\begin{align}\label{energy:boot}
  \norm{\dx^{n}\ue}_{L^2}
        \le C(k_G(t))\exp\left(\int_0^t C(\tld M(t')) dt'\right)
        \left( \norm{\dx^{n}(u_{0})_{\delta}} + \norm{\tld f_n(\dx^{n-1} \vue)}_{L^1_t L^2}\right).
\end{align}
\begin{rem}\label{coeff-n}
  To emphasize the dependence of the constants $k_G$ and $\tld M$ upon the coefficients of \eqref{nonlinear-linear} and hence implicitly on  $\dx^n\vue$ we will add superscripts of $n$, such as $k^n_G$ and $\tld M^n$. 
  That is, we will denote the coefficient norms in Theorem \ref{thm:lin} for \eqref{nonlinear-linear} as $k^n_G$ and $\tld M^n$.
\end{rem}

\subsection{Remainder terms and coefficient estimates}
Estimate \eqref{energy:boot} is a crucial ingredient for the proof of Proposition \ref{uniform:est}. We have
essentially reduced the estimate of the $H^7$-norm and the $H^8$-norm to a proper estimate of coefficients captured by $k_G$ and $\tld M$, as well as by the lower order terms that we denote by $\tld f_n$. We begin with the estimate of the lower order terms $\tld f_n$ for $n=7$ and $n=8$.%

\begin{lem}\label{lem:remainder}%
  Let $\tld f_n$ be as in \eqref{differentiate-nonlinear-0}. Then the following bounds are satisfied:
  \begin{align}\label{diff:lower-terms}
  \begin{split}
     & \norm{\tld f_7}_{L^\infty_tL^2_x} \le C(\norm{\ue}_{H^7}),\\%
      & \norm{\tld f_8}_{L^\infty_tL^2_x} \le C(\norm{\ue}_{H^{7}})\norm{\ue}_{H^8}.%
  \end{split}
     \end{align}
\end{lem}
\begin{proof}
Since $\tld f_7(\vo,x,t)=0$, we can use the Fundamental Theorem of Calculus as follows:
$$\tld f_7 = \sum_{j=0}^6\dx^j\ue\int_0^1 \partial_{z_j}\tld f_7(s\dx^6\vue) ds.$$
  Since $\tld f_7 \in C^1_t C^4_z W^{4,\infty}_{x},$ we may apply \eqref{coeff-bdd} to the integrands, concluding
  \begin{align*}
    \left|\int_0^1 \partial_{z_j} f(s\dx^6\vue) ds\right| \le C(\norm{\ue}_{W^{6,\infty}}).
  \end{align*}
  Therefore
  \begin{align*}
    \norm{\tld f_7}_{L^\infty_tL^2_x} \le C(\norm{\ue}_{W^{6,\infty}})\norm{\ue}_{H^6},
  \end{align*}
and  Sobolev embedding, i.e. the inclusion $H^1(\R)\subseteq L^\infty(\R),$ allows us to conclude
  \eqref{diff:lower-terms} for $n=7$.\\

  Now for $n= 8$, we  put fewer derivatives in $L^\infty$ by using finer structure of $\tld f_8$. Namely, from \eqref{f-n} we obtain
  \begin{align*}
    \tld f_8 = \sum_{j=-1}^5 [\partial_{z_j}\tld f_7  + \partial_{z_j} a_{0,6}\dx^7\ue] \dx^{j+1}\ue + \partial_{z_6} \tld f_7 \dx^7\ue+\partial_{z_6}a_{0,6} \dx^7\ue^{2}.
  \end{align*}
  Estimating the $L^2$-norm of the expression above, we obtain
  \begin{align*}
    \norm{\tld f_8}_{L^2_x} \le \sum_{j=-1}^5 \norm{\partial_{z_j}\tld f_7  + \partial_{z_j} a_{0,6}\dx^7\ue}_{L^\infty_{x}} \norm{\dx^{j+1}\ue}_{L^2_x} + \norm{\partial_{z_6} \tld f_7 +\partial_{z_6}a_{0,6} \dx^7\ue}_{L^\infty} \norm{ \dx^7\ue}_{L^2}.
  \end{align*}
  Since both $\tld f_7$ and $a_{0,6}$ are in $C^1_t C^4_z W^{4,\infty}_x$, we may further estimate coefficients by
   \eqref{coeff-bdd}, yielding
  \begin{align*}
    \norm{\tld f_8}_{L^2_x} \le C(\norm{\ue}_{W^{6,\infty}})\cdot (1+\norm{\ue}_{W^{7,\infty}})\norm{\ue}_{H^7},
  \end{align*}
  where we have used that no more than one factor of $\dx^7\ue$ is present in the $L^\infty$ terms. Using Sobolev embedding we conclude \eqref{diff:lower-terms} for $n=8$.
\end{proof}
We now estimate the coefficients for \eqref{nonlinear-linear}.
\begin{lem}\label{lem:coeff} Let $\ue \in C^0_t H^8$ for $\eps>0$ or $\ue\in C^t_0 H^7$ for $\eps=0$ satisfy \eqref{parab:non} and consider the linear equation \eqref{nonlinear-linear}.
Then the coefficient norms from \eqref{linear-coeff} for the equation \eqref{nonlinear-linear} can be estimated as follows (where we follow the convention of the Remark \ref{coeff-n}):
  \begin{align}
    k_G^n(t) \le C(k(t),n),\label{coeff-bdd-est}\\
  \tld M^n(t) \le C(M_\eps(t),k(t),n),  \label{lin:coeff}
  \end{align}
 where $k(t)$ and $M_\eps$ are as in \eqref{dispersion+low-norm} and \eqref{high-norm} respectively.

\end{lem}
\begin{rem}
  Note that Lemma \ref{lem:coeff} and Theorem \ref{thm:lin} determine the regularity we pursue in Theorem \ref{main-thm};
  that is, these are the steps of our proof which cause us to work in the space $H^{7}.$
  Knowing more precise structure of  the function $f$ in \eqref{generalEquation}, e.g. if $f$ is ``less nonlinear,''
  would lower the regularity needed in our proof.  In particular, our argument could utilize the spaces $H^4$ for $K(2,2)$-type equations and $H^{\frac{3}{2}^+}$ for KdV.
\end{rem}
\begin{proof}
We first estimate the lower-order norms for $k_G$. The estimates of the dispersive coefficient follow from
the coefficient hypothesis (A1) and the lower bound on the nonlinear dispersion \eqref{Dispersion}. %

  More precisely, from the definition of the coefficient $a_3 = f_{z_3}(\dx^3 \vue,\ldots),$ we see
  $$
\norm{a_3}_{L^\infty} +   \left\|\frac{1}{a_3}\right\|_{L^\infty} \le \norm{f_{z_3}}_{L^\infty} +
\left\|\frac{1}{f_{z_3}}\right\|_{L^\infty}.
  $$
Using \eqref{coeff-bdd} for $f_{z_3}(\dx^3 \vue),$  the definition of $\lambda(t)$ in \eqref{Dispersion},
and Sobolev embedding implies
 \begin{equation}\nonumber
   \norm{a_3}_{L^\infty}+ \left\|\frac{1}{a_3}\right\|_{L^\infty}  \le C(\norm{\vue}_{W^{3,\infty}})+\frac{1}{\lambda(t)}
   \le  C\left(\norm{\vue}_{H^4},\frac{1}{\lambda(t)}\right).
 \end{equation}
 It remains to estimate $\norm{\int_0^x\frac{a_2}{a_3} dx'}_{L^\infty_x}$ to finish \eqref{coeff-bdd-est} for $\tld k_G$.
  From the definition of the coefficients in \eqref{differentiate-nonlinear-0}, we have
  \begin{align*}
      \int_0^x\frac{a_2}{a_3} dx' = n \log f_{z_3}(\dx^3 \vue,\ldots) + \int_0^x g_M dx'.
  \end{align*}
  Observe that a logarithm is dominated by its argument:
  \begin{align*}
    \log y \le y+\frac{1}{y},\quad \text{ for } y>0.
  \end{align*}
  Hence the logarithm is comparable to the norm previously estimated above:
  \begin{align*}
    |\log f_{z_3}(\dx^3 \vue)| \le  \norm{a_3}_{L^\infty}+ \left\|\frac{1}{a_3}\right\|_{L^\infty}.
  \end{align*}
Meanwhile, from \eqref{Diffusion} we have
  $$
  \int_0^x g_M dx' = g_D(\dx^2\vue) + \int_0^x g_H(\dx^3\vue)dx'.
  $$
The term  $g_D$ is controlled by \eqref{coeff-bdd}: $$\norm{g_D}_{L^\infty} \le C(\norm{\vue}_{W^{2,\infty}}).$$
Continuing, the Taylor expansion of $g_H$ to the quadratic terms using \eqref{g_H-is-cubic} implies
  \begin{align*}
    g_H(\dx^3\vue) = \sum_{i,j=0}^3 \dx^i \ue \dx^j \ue \int_0^1\int_0^1\partial_{z_i,z_j}g_{H}(s_1s_2\dx^3\vue)s_1ds_2ds_1.
  \end{align*}
  We then use Cauchy-Schwarz, \eqref{coeff-bdd}, and Sobolev embedding:
  \begin{align*}
   \left\|\int_0^x g_H(\dx^3\vue)\right\|_{L^\infty}& \le \sum_{i,j=0}^3 \norm{\dx^i \ue}_{L^2}\norm{ \dx^j \ue }_{L^2}\sup_{|s|\le 1}\norm{\partial_{z_i,z_j} g_{H}(s\dx^3\vue)}_{L^\infty_x}\\
   & \le C(\norm{\ue}_{H^4}).
  \end{align*}
   Estimates of $\tld M(t)$ are quite similar to estimates of $\tld k_G(t)$. The relevant coefficients in
   \eqref{differentiate-nonlinear-0} can be written in the form $a_j=a_{j,n}(\dx^{6-j}\ue)$ for $a_j$ satisfying $a_j\in C^1_t C^{7+j}_z W^{7+j,\infty}_x$ by \eqref{higher-coeff} and (A1). Thus
   \begin{align*}
     \norm{a_j}_{W^{3-j,\infty}}\le C(\norm{\dx^3\vue}_{W^{3-j,\infty}_x})\le C(\norm{\ue}_{H^7}) = C(M_\eps(t)).
   \end{align*}
 We then observe that $\dt[ a_3(x,t)] = f_{z_3,t}+\sum_{j=0}^3f_{z_3,z_j}(\dx^3 \vue)\dt \dx^j \ue$. Estimating as above,
 and using Proposition \ref{lem:dt},
 \begin{align*}
   \norm{\dt a_3}_{L^\infty} \le C(\norm{\dt \ue}_{H^4},\norm{\ue}_{H^4}) \le C(M_\eps(t)).
 \end{align*}

Meanwhile, differentiation yields the following:
 \begin{align*}
     \dt \int_0^x\frac{a_2}{a_3} dx' = n \dt[\log f_{z_3}(\dx^3 \vue)] + \int_0^x\dt [g_M ]dx'.
  \end{align*}
  For the first term, $\dt[\log f_{z_3}(\dx^3 \vue)] \le C(\norm{\dt a_3}_{L^\infty},\norm{a_3}_{L^\infty},\norm{\frac{1}{a_3}}_{L^\infty} )$. Then $\dt[g_M]$ is estimated similarly to $\int_0^x\frac{a_2}{a_3} dx'$, with the additional $\dt\dx^j \ue$
  terms estimated with Proposition \ref{lem:dt}.  These considerations yield the following:
\begin{align*}
 \left\| \dt \int_0^x\frac{a_2}{a_3} dx'\right\|_{L^\infty_x} \le
 C\left(\norm{\ue}_{H^7},\eps\norm{\ue}_{H^8},\frac{1}{\lambda(t)}\right).
\end{align*}
This completes the proof.
\end{proof}

\subsection{Proof of Proposition \ref{uniform:est}}\label{subsec:uniform-proof}
\begin{proof}
  By applying Lemma \ref{lem:coeff}, we conclude from \eqref{energy:boot} for $n\ge 7$ that
  \begin{align*}
     \norm{\dx^{n}\ue(t)}_{L^2}
        \le C(k(t),n)\exp\left(\int_0^t C( M_\eps(t'),k(t')) dt'\right)
        \left( \norm{\dx^{n}(u_{0})_{\delta}} + t\norm{\tld f_n(\dx^{n-1} \vue)}_{L^\infty_t L^2}\right).
  \end{align*}
  Adding this estimate to \eqref{dnorm-est1} we obtain
  \begin{align*}
    \norm{\ue(t)}_{H^n} \le C(k(t))e^{tC\left(\sup_{t'\le t} \left[M_\eps(t'),k(t')\right]\right)}\left(\norm{u_0}_{H^n}+tC(M_\eps(t)) + t\norm{\tld f_n(\dx^{n-1} \vue)}_{L^\infty_t L^2}\right).
  \end{align*}
  Using Lemma \ref{lem:remainder} for $\tld f_7$ implies
  \begin{align*}
    \norm{\ue(t)}_{H^7} \le  C(k(t))e^{tC\left(\sup_{t'\le t} \left[M_\eps(t'),k(t')\right]\right)}\left(\norm{u_0}_{H^n}+tC(M_\eps(t)) + tC(\norm{\ue(t')}_{L^\infty_{t'}H^7})\right).
  \end{align*}
  Meanwhile the use of \eqref{diff:lower-terms} for $\tld f_8$ gives
  \begin{multline}\label{remainder-H8}
     \eps\norm{\ue(t)}_{H^8} \le  C(k(t))e^{tC\left(\sup_{t'\le t} \left[M_\eps(t'),k(t')\right]\right)}\\
      \cdot\left(\eps \norm{u_0}_{H^8}+\eps tC(M_\eps(t))
      + tC(\norm{\ue(t')}_{L^\infty_{t'}H^7}\cdot \eps\norm{\ue(t')}_{L^\infty_{t'}H^8}\right).
\end{multline}
  Adding the last two estimates concludes the proof.
\end{proof}

\subsection{Refined boundedness}
The following lemma is not needed for the proof of the Proposition \ref{uniform:est}, for which the estimates
of $\tld f_7$ and $\tld f_8$ are enough. However, in order to justify the $C_t H^7_x$ regularity of the solution we need a more precise estimate of $\tld f_{11}$. We can see this effect quite well in \eqref{remainder-H8}, where an estimate $\norm{\tld f_8}_{L^2}\le C(\norm{\ue}_{H^8})$ would not be sufficient. Thus, the lemma below can be thought of as a refinement of Lemma \ref{lem:remainder}.
\begin{lem}\label{lem:remainder11}
  Let $\ue$ be a solution of \eqref{parab:non} and let $\tld f_{11}$ be as in \eqref{differentiate-nonlinear-0} for $n=11$. Then
  \begin{align}  \label{remainder-H11}
    \norm{\tld f_{11}}_{L^2}\le C(\norm{\ue}_{H^7}) \norm{\ue}_{H^{11}}.
  \end{align}
\end{lem}
This lemma would allow us to show
 persistence of regularity, i.e. a solution with $H^{11}$ data has an $H^{11}$ solution on the same time interval. We defer the proof of Lemma \ref{lem:remainder11} until we prove the following corollary, which is the main motivation for the lemma.
  \begin{coro}\label{coro:remainder}
     For $\ue$ as before and $M$, $T$ from Corollary \ref{uniform:exist}, there exists a constant $C=C(M,k(0))$, such that
     \[ \norm{\ue}_{L^\infty_{[0,T]} H^{11}} \le C\delta^{-4}.\]
    \end{coro}
\begin{proof}
  As $\ue$ satisfies \eqref{parab:non}, $\dx^{11}\ue$ satisfies \eqref{differentiate-nonlinear-0} for $n=11$. We now apply the linear estimate, Theorem \ref{thm:lin}, with the coefficients for $\dx^{11}\ue$ as in \eqref{nonlinear-linear}.
  First, observe that coefficients in \eqref{differentiate-nonlinear-0} for $n=11$ satisfy the same bounds as those for $\dx^7\ue$ by Lemma \ref{lem:coeff}. Second, observe that we can use Theorem \ref{thm:lin} on an interval $[t_0,t_1]\subset [0,T]$ rather than $[0,t_1]$. With those two observations in mind we get, for $t_0\le t'\le t_1,$
  \begin{align*}
    \norm{\dx^{11}\ue(t')}_{L^2}\le C(M(t),k(t))(\norm{\ue(t_0)}_{H^{11}}+\norm{\tld f_{11}}_{L^1_{[t_0,t_1]} L^2_x}).
  \end{align*}
  We extract the time factor and estimate $\tld f_{11}$ in $L^\infty$ in time.
  Furthermore, Lemma \ref{lem:remainder11} allows us to estimate the remainder $\tld f_{11}$ with no more than a single factor of $\dx^{11}\ue$:
  \begin{align*}
    \norm{\dx^{11}\ue(t)}_{L^2}\le C(M(t),k(t))
     (\norm{\ue(t_0)}_{H^{11}} +(t_1-t_0)C(\norm{\ue}_{L^\infty_{[t_0,t_1]}H^7})\norm{\ue(t_0)}_{L^\infty_{[t_0,t_1]}H^{11}}).
  \end{align*}
  As $[t_0,t_1]\subset [0,T]$, $\norm{\ue(t')}_{H^7}\le M$ for $t'\in [t_0,t_1]$. Incorporating this estimate, after adding the $L^2$ norm, we get
  \begin{multline}\nonumber
    \norm{\ue(t')}_{H^{11}}\le C(\norm{\ue}_{L^2}+\norm{\dx^{11}\ue}_{L^2})\\
    \le C(M)+ C(M,k(t'))\norm{\ue(t_0)}_{H^{11}} + (t_1-t_0)C(M,k(t'))\norm{\ue(t_0)}_{L^\infty_{[t_0,t_1]}H^{11}}.
  \end{multline}
  By Remark \ref{rem:k}, $k(t')\le 4k(0)$. We can thus eliminate the dependence of the bound on $t'$ at a cost of a larger constant:
  \begin{equation}\nonumber
     \norm{\ue(t')}_{H^{11}}\le C(M)+ C(M,k(0))\norm{\ue(t_0)}_{H^{11}}
      + (t_1-t_0)C(M,k(0))\norm{\ue(t_0)}_{L^\infty_{[t_0,t_1]}H^{11}}.
  \end{equation}
   Furthermore, we let $t_1=t_0+\triangle t$ and we make the width of the interval $\triangle t$ small enough so that
  \begin{align*}
    (t_1-t_0) C(M,k(0))= \frac{1}{2}.
  \end{align*}
  This choice allows us to eliminate the $H^{11}$ term on the right hand side:
  \begin{align*}
    \norm{\ue(t_0)}_{L^\infty_{[t_0,t_0+\triangle t]}H^{11}}\le C(M)+ C(M,k(0))\norm{\ue(t_0)}_{H^{11}}.
  \end{align*}
  We can now iterate this estimate for $t_0=0$, $\triangle t$, $2\triangle t$,\ldots, $j\triangle t$, where
  \begin{align*}
    j\triangle t \ge T \text{ for $T$ from Corollary \ref{apriori}}
  \end{align*}
  so that we get
  \begin{align*}
    \norm{\ue(t_0)}_{L^\infty_{[0,t_0+j\triangle t]}H^{11}} \le C(M,k(0))^j(1+\norm{\ue(0)}_{H^{11}}).
  \end{align*}
  Using Lemma \ref{BS} implies that $\norm{\ue(0)}_{H^{11}}\le C(M)\delta^{-4}$ and concludes the proof.
\end{proof}
We now return to the proof of Lemma \ref{lem:remainder11}. In the proof we need the following variation of a basic interpolation result.
\begin{lem}\label{interpolation}
  Let $w\in H^{11}$. Then for $0\le\theta\le 4,$
  \begin{align*}
    \norm{w}_{H^{7+\theta}}\le \norm{w}_{H^{11}}^{\frac{\theta}{4}} \norm{w}_{H^7}^{\frac{3\theta}{4}}.
  \end{align*}
\end{lem}
\begin{proof}
  Use the Plancherel Theorem and use H\"{o}lder's inequality for the function $\hat v(\xi)=\jap{\xi}^7\hat w(\xi):$
  \begin{align*}
    \int \jap{\xi}^{2\theta} |\hat v(\xi)|^2 d\xi\le \left(\int\jap{\xi}^{2\cdot 4} |\hat v(\xi)|^2d\xi\right)^\frac{\theta}{4}
    \left(\int |\hat v(\xi)|^2d\xi\right)^\frac{3\theta}{4}.
  \end{align*}
  This concludes the proof.
\end{proof}
\begin{proof}[Proof of Lemma \ref{lem:remainder11}]
  We use the precise variant from the Corollary \ref{lem:remainder11-prep}:
  \begin{align*}
    \tld f_{11}= \sum_{k=2}^{11}\sum_{i_l,j_l\in \mathcal{S}_k^{11}}
    C_{\vec{i},\vec{j}}\cdot f_{j_1,\ldots j_k}(z)\dx^{i_1}z_{j_1}\cdots \dx^{i_k} z_{j_k}\mid_{z=(\dx^3\vue,x,t)},
  \end{align*}
  where the $k-$tuples $\vec{i}$, $\vec{j}$ of admissible indices $\mathcal{S}_k$ are defined by the following:
  \begin{align*}
    \mathcal{S}_k^{11} =\left\{1\le l \le k; i_1\ge i_2\ge\ldots\ge i_k\ge 1;  \sum_{l=1}^k i_l=11; -1\le j_l\le 3; i_l+j_l<11;\max\{j_l,1-i_l\}\ge 0 \right\},
  \end{align*}
  i.e. we include all terms in $\dx^{11}[ f]$ that are of order less than $11$ in $u$. We then place the highest-order term in $L^2$, remaining terms in $L^{\infty},$ and analyze four different scenarios:
  \begin{align*}
    \norm{\tld f_{11}(\dx^{10}\vue)}_{L^2}\le &
    \sum_{k=2}^{11}
  \sum_{i_l,j_l\in \mathcal{S}_k}
    C_{\vec{i},\vec{j}} \norm{ f_{j_1,\ldots j_k}(z)}_{L^\infty}\norm{\dx^{i_1}z_{j_1}}_{L^2}
    \norm{\dx^{i_2}z_{j_2}}_{L^\infty}\cdots \norm{\dx^{i_k} z_{j_k}}_{L^\infty}\\
    & :=I_{\ge 8} + I_7 + I_6 +  I_{\le 5}.
  \end{align*}
  Here, the sum is separated by the largest number of derivatives $i_1$. That is, $I_{\ge 8}$ includes all the terms where $i_1\ge 8$; $I_7$, where $i_1=7$, etc... We estimate the new $I_l$ sums term by term.\\

  For any sum $I_{\le 5}$ through $I_{\ge 8}$, the fact that the number of $i_l$ derivatives adds up to 11 means that $k-1\le \sum_{l=2}^k i_l \le 11-i_1$. In particular, for $I_{\ge 8}$, $ k\le 4$ and $i_2\le 3$. Therefore all the $L^\infty$ terms have at most $3$ derivatives. We also estimate $f_{\vec{j}}$ via \eqref{coeff-bdd}:
  \begin{align*}
    I_{\ge 8} \le \sum_{i_1=8}^{10}\sum_{i_1+j_1<11} C(\norm{z}_{L^\infty_x}) (1+\norm{z}_{W^{3,\infty}_x}^4)\norm{\dx^{i_1} z_{j_1}}_{L^2}.
  \end{align*}
  Using Sobolev embedding and $z=(\dx^3\vue,x,t)$ we get
  \begin{align*}
    I_{\ge 9} \le C(\norm{\ue}_{H^7})\norm{\ue}_{H^{11}}.
  \end{align*}
  For $I_7$, $i_1=7$, which allows $i_2\le 4$ and $i_l\le 11-i_1-i_2\le 3$ for $l\ge3$. Hence by \eqref{coeff-bdd} as before,
  \begin{align*}
    I_7 \le C(\norm{z}_{L^\infty})\cdot \norm{\dx^7 z}_{L^2}\cdot (1+\norm{z}_{W^{4,\infty}})C(\norm{z}_{W^{3,\infty}}).
  \end{align*}
  Here, we have estimated terms beyond $i_2$ with $W^{3,\infty}_x$ norm. Hence by Sobolev embedding,
  \begin{align*}
    I_7 \le C(\norm{\ue}_{H^7})\norm{\ue}_{H^{10}}\cdot\norm{\ue}_{H^8}.
  \end{align*}
  We now use Lemma \ref{interpolation} to conclude
  \begin{equation}\nonumber
    I_7 \le C(\norm{\ue}_{H^7}) (\norm{\ue}_{H^{11}}^{\frac{3}{4}}\cdot \norm{\ue}_{H^{7}}^{\frac{1}{4}})\cdot (\norm{\ue}_{H^{11}}^{\frac{1}{4}}\norm{\ue}_{H^{7}}^{\frac{3}{4}})
    \le C(\norm{\ue}_{H^7})\norm{\ue}_{H^{11}}.
  \end{equation}
  The remaining terms are similar. For $I_6$, we have $i_2\le 5$. If $i_2=5$, then $k=2$
  as we do not have any derivatives left for $i_3$, etc... If $i_2\le 4$, then $i_3\le \min (5-i_2,i_2) \le 2$. Note that we used the non-increasing arrangement $i_1\ge i_2\ge i_3\cdots$. Hence
  \begin{align*}
    I_6 \le C(\norm{z}_{L^\infty})\cdot \norm{\dx^6 v}_{L^2}\cdot\norm{z}_{W^{5,\infty}}\cdot C(\norm{z}_{W^{2,\infty}}).
  \end{align*}
  Hence by Sobolev embedding and Lemma \ref{interpolation} we can conclude with
  \begin{equation}\nonumber
    I_6\le C(\norm{\ue}_{H^7})\cdot \norm{\ue}_{H^9}^2
    \le C(\norm{\ue}_{H^7})\norm{\ue}_{H^{11}}.
  \end{equation}
  For $I_{\le 5}$, we have $1\le i_2\le i_1\le 5$, thus $$1\le i_3\le \min\{6-i_1-i_2,i_2\} \le 3.$$
  We estimate
  \begin{multline}\nonumber
    I_{\le 5}\le C(\norm{z}_{L^\infty})\norm{z}_{H^5}\norm{z}_{W^{5,\infty}}C(\norm{z}_{W^{3,\infty}})\\
    \le C(\norm{\ue}_{H^7}) \norm{\ue}_{H^8}\norm{\ue}_{H^9}
    \le  C(\norm{\ue}_{H^7})\norm{\ue}_{H^{11}}.
  \end{multline}
  Adding the estimates for $I_{\ge 7}$,\ldots $I_{\le 5}$ completes the proof.
  \end{proof}

\section{Passage to the limit}

   \begin{prop}\label{lem:diff}
  Suppose that $ u_\eps $ and $\uee$ are in $C_T H^8$ (or $H^7$ for $\eps,\eps'=0$) and each solve the evolution equation
  \eqref{parab:non}, with
  initial data $u^{0}$ and $u^{0'},$ respectively. Then for $M=\displaystyle\sup_{\tau\in[0,t]}\left(M_\eps(\tau) + k(\tau)\right)$, there is a  constant
  $C(M)$ such that
  \begin{align}\label{lem:diff:eq}
    \norm{u_{\varepsilon}(t) - u_{\varepsilon'}(t)}_{L^\infty_T H^3_x}
    \le C(M)T\norm{u_{\varepsilon}(t) - u_{\varepsilon'}(t)}_{L^\infty_T H^3_x} + (\eps+\eps')C(M)
    +C(M)\|u^{0}-u^{0'}\|_{H^{3}_{x}}.
    \end{align}
    Therefore, there exists $T_{1}>0$ such that
    \begin{align}\label{diff:lim}
      \norm{u_{\varepsilon} - u_{\varepsilon'}}_{L^\infty_{T_1 H^3_x}}\le 2(\eps+\eps')C(M)+C(M)\|u^{0}-u^{0'}\|_{H^{3}_{x}}.
    \end{align}
\end{prop}

\begin{rem}
The data $u^{0}$ and $u^{0'}$ will be taken to depend on smoothing parameters $\delta$ and $\delta',$ but for
the present proposition, it is not necessary to be that explicit about the nature of the data.
\end{rem}

\begin{rem}
If we take $\varepsilon=\varepsilon'=0$ and $u^{0}=u^{0'}$ in \eqref{diff:lim}, then we see that solutions are unique.
This proves one of the claims of Theorem \ref{main-thm}.
\end{rem}
\begin{rem}
  By iterating \eqref{diff:lim} as in the proof of Corollary \ref{coro:remainder}, we can replace $T_1$ with $T$ from Corollary \ref{uniform:exist}. This reiterates that size of the solution and dispersion (i.e. $M_0(t)$ and $k(t)$ from \eqref{high-norm} and \eqref{dispersion+low-norm}) determine the time of wellposedness.
\end{rem}
\begin{proof}
We let $0\leq\varepsilon'\leq\varepsilon,$ and we consider the solutions
$u_{\varepsilon}$ and $u_{\varepsilon'}$ which we have shown
above to exist.
We treat the difference $u_{\varepsilon}-u_{\varepsilon'}$ in $H^{3}$ by treating the difference first in $L^{2},$
and then by treating three spatial derivatives in $L^{2}.$
We first note that the inequality
\begin{equation}\nonumber
\|u_{\varepsilon}-u_{\varepsilon'}\|_{L^{\infty}_{T}L^{2}_{x}}
\leq C(M)T\|u_{\varepsilon}-u_{\varepsilon'}\|_{L^{\infty}_{T}H^{3}_{x}}
+(\varepsilon+\varepsilon')C(M)+C(M)\|u^{0}-u^{0'}\|_{H^{3}_{x}}
\end{equation}
follows immediately from the Fundamental Theorem of Calculus (integrating $(u_{\varepsilon}-u_{\varepsilon'})_{t}$
with respect to time) and a Lipschitz estimate for the function $f.$  We therefore are free to move on to
considering three derivatives of the difference.

We define $w$ to be three $x$-derivatives of the difference of the solutions,
\begin{equation}\nonumber
w=\partial_{x}^{3}\left(u_{\varepsilon}-u_{\varepsilon'}\right).
\end{equation}
We can write the evolution equation for $w$ in the framework of \eqref{lin}; to confirm this,
we will identify explicitly, more or less, all of the coefficients and the forcing.
Some additional notation will help with this task, so we will have the following decompositions:
\begin{equation}\nonumber
h=h_{-1}+h_{0}+h_{1}+h_{2}+h_{3}+h_{4},
\end{equation}
\begin{equation}\nonumber
a_{1}=b_{1}+b_{2},
\end{equation}
\begin{equation}
a_{0}=d_{0}+d_{1}+d_{2}.
\end{equation}
The coefficients $a_{2}$ and $a_{3}$ will be more straightforward to calculate, and no such
decomposition will be necessary for them.
Furthermore, we introduce the following notation:
\begin{equation}\nonumber
\partial_{t}w=\partial_{t}\partial_{x}^{3}u_{\varepsilon}-\partial_{t}\partial_{x}^{3}u_{\varepsilon'}
=A_{2}+A_{3}+A_{4}+A_{5}+A_{6}+A_{7}.
\end{equation}
Here, $A_{7}$ will consist of terms which involve seventh
derivatives of $u;$ this term simply comes from the parabolic regularization.
Continuing, $A_{6}$ will consist of terms from the right-hand side of \eqref{bigThreeDerivatives} which
involve sixth derivatives of $u,$ $A_{5}$ will consist of terms which involve fifth derivatives of $u,$
$A_{4}$ will consist of terms which involve fourth derivatives of $u$ but no higher derivatives,
$A_{3}$ will consist of terms which involve third derivatives of $u$ but no higher derivatives,
and $A_{2}$ will be the remaining terms, which involve at most second derivatives of $u.$

To begin, we may write $A_{7}$ simply as
\begin{equation}\nonumber
A_{7}=\varepsilon'\partial_{x}^{7}u_{\varepsilon'}-\varepsilon\partial_{x}^{7}u_{\varepsilon}.
\end{equation}
We add and subtract, to form the fourth-derivative term on the right-hand side of \eqref{lin}:
\begin{equation}\nonumber
A_{7}=-\varepsilon(\partial_{x}^{7}u_{\varepsilon}-\partial_{x}^{7}u_{\varepsilon'})
-\varepsilon\partial_{x}^{7}u_{\varepsilon'}+\varepsilon'\partial_{x}^{7}u_{\varepsilon'}
=-\varepsilon\partial_{x}^{4}w+(\varepsilon'-\varepsilon)\partial_{x}^{7}u_{\varepsilon'}.
\end{equation}
The second term on the right-hand side makes up the contribution $h_{4}:$
\begin{equation}\nonumber
h_{4}=(\varepsilon'-\varepsilon)\partial_{x}^{7}u_{\varepsilon'}.
\end{equation}

We next note that on the right-hand side
of \eqref{bigThreeDerivatives}, there is only one term that is a sixth derivative of $u;$ this term contributes
the following to the evolution equation for $w:$
\begin{multline}\nonumber
A_{6}= f_{z_{3}}(\dx^{3}{\bf u_{\eps}})(\partial_{x}^{6}u_{\varepsilon})
-f_{z_{3}}(\dx^{3}{\bf u_{\eps'}})(\partial_{x}^{6}u_{\varepsilon'})
\\
=f_{z_{3}}(\dx^{3}{\bf u_{\eps}})(\partial_{x}^{6}u_{\varepsilon})
-f_{z_{3}}(\dx^{3}{\bf u_{\eps'}})(\partial_{x}^{6}u_{\varepsilon})
+f_{z_{3}}(\dx^{3}{\bf u_{\eps'}})(\partial_{x}^{6}u_{\varepsilon})
-f_{z_{3}}(\dx^{3}{\bf u_{\eps'}})(\partial_{x}^{6}u_{\varepsilon'})
\\
=\left\{\left(f_{z_{3}}(\dx^{3}{\bf u_{\eps}})-f_{z_{3}}(\dx^{3}{\bf u_{\eps'}})\right)
(\partial_{x}^{6}u_{\varepsilon})
\right\}+f_{z_{3}}(\dx^{3}{\bf u_{\eps'}})(\partial_{x}^{3}w).
\end{multline}
We define $a_{3}$ and $h_{3}$ as follows:
\begin{equation}\label{difference-a3}
a_{3}=f_{z_{3}}(\dx^{3}{\bf u_{\eps'}}),
\end{equation}
\begin{equation}\nonumber
h_{3}=\left(f_{z_{3}}(\dx^{3}{\bf u_{\eps}})-f_{z_{3}}(\dx^{3}{\bf u_{\eps'}})\right)
(\partial_{x}^{6}u_{\varepsilon}).
\end{equation}

To identify $a_{2}$, we must consider the six terms on the right-hand side of \eqref{bigThreeDerivatives}
which involve fifth spatial derivatives of the unknown, $u:$
\begin{multline}\nonumber
A_{5}=\Big(3f_{xz_{3}}(\dx^{3}{\bf u_{\eps}})+3f_{z_{0}z_{3}}(\dx^{3}{\bf u_{\eps}})(\partial_{x}u_{\varepsilon})
+3f_{z_{1}z_{3}}(\dx^{3}{\bf u_{\eps}})(\partial_{x}^{2}u_{\varepsilon})
+3f_{z_{2}z_{3}}(\dx^{3}{\bf u_{\eps}})(\partial_{x}^{3}u_{\varepsilon})
\\
+3f_{z_{3}z_{3}}(\dx^{3}{\bf u_{\eps}})(\partial_{x}^{4}u_{\varepsilon})
+f_{z_{2}}(\dx^{3}{\bf u_{\eps}})\Big)(\partial_{x}^{5}u_{\varepsilon})
\\
-\Big(3f_{xz_{3}}(\dx^{3}{\bf u_{\eps'}})+3f_{z_{0}z_{3}}(\dx^{3}{\bf u_{\eps'}})(\partial_{x}u_{\varepsilon'})
+3f_{z_{1}z_{3}}(\dx^{3}{\bf u_{\eps'}})(\partial_{x}^{2}u_{\varepsilon'})
+3f_{z_{2}z_{3}}(\dx^{3}{\bf u_{\eps'}})(\partial_{x}^{3}u_{\varepsilon'})
\\
+3f_{z_{3}z_{3}}(\dx^{3}{\bf u_{\eps'}})(\partial_{x}^{4}u_{\varepsilon'})
+f_{z_{2}}(\dx^{3}{\bf u_{\eps'}})\Big)(\partial_{x}^{5}u_{\varepsilon'}).
\end{multline}
After some adding and subtracting, we can write this as
\begin{equation}\nonumber
A_{5}=a_{2}(\partial_{x}^{2}w)+b_{2}(\partial_{x}w)+d_{2}w+h_{2},
\end{equation}
where we have the following formulas:
\begin{multline}\label{difference-a2}
a_{2}=3f_{xz_{3}}(\dx^{3}{\bf u_{\eps'}})+3f_{z_{0}z_{3}}(\dx^{3}{\bf u_{\eps'}})(\partial_{x}u_{\varepsilon'})
+3f_{z_{1}z_{3}}(\dx^{3}{\bf u_{\eps'}})(\partial_{x}^{2}u_{\varepsilon'})
+3f_{z_{2}z_{3}}(\dx^{3}{\bf u_{\eps'}})(\partial_{x}^{3}u_{\varepsilon'})
\\
+3f_{z_{3}z_{3}}(\dx^{3}{\bf u_{\eps'}})(\partial_{x}^{4}u_{\varepsilon'})
+f_{z_{2}}(\dx^{3}{\bf u_{\eps'}}),
\end{multline}
\begin{equation}\label{b2Formula}
b_{2}=3f_{z_{3}z_{3}}(\dx^{3}{\bf u_{\eps'}})(\partial_{x}^{5}u_{\varepsilon}),
\end{equation}
\begin{equation}\nonumber
d_{2}=3f_{z_{2}z_{3}}(\dx^{3}{\bf u_{\eps'}})(\partial_{x}^{5}u_{\varepsilon}),
\end{equation}
and
\begin{multline}\nonumber
h_{2}=\Big\{(3f_{xz_{3}}(\dx^{3}{\bf u_{\eps}})+3f_{z_{0}z_{3}}(\dx^{3}{\bf u_{\eps}})(\partial_{x}u_{\varepsilon})
+3f_{z_{1}z_{3}}(\dx^{3}{\bf u_{\eps}})(\partial_{x}^{2}u_{\varepsilon})
+f_{z_{2}}(\dx^{3}{\bf u_{\eps}}))
\\
-(3f_{xz_{3}}(\dx^{3}{\bf u_{\eps'}})+3f_{z_{0}z_{3}}(\dx^{3}{\bf u_{\eps'}})(\partial_{}u_{\varepsilon'})
+3f_{z_{1}z_{3}}(\dx^{3}{\bf u_{\eps'}})(\partial_{x}^{2}u_{\varepsilon'})
+f_{z_{2}}(\dx^{3}{\bf u_{\eps'}}))\Big\}(\partial_{x}^{5}u_{\varepsilon})
\\
+\left(3f_{z_{2}z_{3}}(\dx^{3}{\bf u_{\eps}})-3f_{z_{2}z_{3}}(\dx^{3}{\bf u_{\eps'}})\right)
(\partial_{x}^{5}u_{\varepsilon})(\partial_{x}^{3}u_{\varepsilon})
+\left(3f_{z_{3}z_{3}}(\dx^{3}{\bf u_{\eps}})-3f_{z_{3}z_{3}}(\dx^{3}{\bf u_{\eps'}})\right)
(\partial_{x}^{5}u_{\varepsilon})(\partial_{x}^{4}u_{\varepsilon}).
\end{multline}

We may continue in this way with $A_{4},$ noting that there are 23 terms from the right-hand side of
\eqref{bigThreeDerivatives} which contribute to $A_{4}.$  We can write
\begin{equation}\nonumber
A_{4}=b_{1}(\partial_{x}w)+d_{1}w+h_{1}.
\end{equation}
We may then treat $A_{3}$ in the same manner, noting that there are 16 terms from the right-hand side of
\eqref{bigThreeDerivatives} which contribute to $A_{3}.$  We may write
\begin{equation}\nonumber
A_{3}=d_{0}w+h_{0}.
\end{equation}
Finally, we note that the remaining terms comprising $A_{2}$ all contribute to $h:$
\begin{equation}\nonumber
A_{2}=h_{-1}.
\end{equation}

Now that we have established
the formulas \eqref{difference-a3} and \eqref{difference-a2}, we can see the following form of the ratio:
\begin{equation}\label{difference-ratio}
\frac{a_{2}}{a_{3}}=\partial_{x}\left(3\ln\left(f_{z_{3}}\left(\dx^{3}{\bf u_{\eps'}})\right]\right)\right)
+\frac{f_{z_{2}}(\dx^{3}{\bf u_{\eps'}})}{f_{z_3}(\dx^{3}{\bf u_{\eps'}})}.
\end{equation}

We seek to apply Theorem \ref{thm:lin} (the linear estimate), and as such, we use the definitions of $k_{G}(t)$
and $\tilde{M}(t)$ as given in \eqref{linear-coeff}.  By Lemma \ref{lem:coeff},
as well as the definition of $a_{3}$ in \eqref{difference-a3},
we see that $\|a_{3}\|_{L^{\infty}_{x}}$ and $\|1/a_{3}\|_{L^{\infty}_{x}}$ are bounded.
To conclude that $k_{G}$ is bounded, we then need to conclude that the antiderivative of $a_{2}/a_{3}$ is bounded;
the antiderivative we must consider, using \eqref{difference-ratio}, is
\begin{equation}\nonumber
\int_{0}^{x}\frac{a_{2}}{a_{3}}(x',t)\ dx' = 3\ln(f_{z_{3}}(\dx^{3}{\bf u_{\eps'}}))(x,t)
-3\ln(f_{z_{3}}(\dx^{3}{\bf u_{\eps'}}))(0,t)+\int_{0}^{x}\frac{f_{z_{2}}(\dx^{3}{\bf u_{\eps'}})}
{f_{z_{3}}(\dx^{3}{\bf u_{\eps'}})}(x',t)\ dx'.
\end{equation}
Again using the definition of $a_{3}$ in \eqref{difference-a3}, and using Lemma \ref{lem:coeff}
and knowing that $a_{3}\in C(\mathbb{R}\times[0,T])$ (this fact also uses Condition (A1)), we see that $a_{3}$ is
positive and bounded away from zero for all $x$ and for all $t\in[0,T].$   The properties of the natural logarithm function
and the bound of Lemma \ref{lem:coeff}
imply a uniform bound for $\ln(f_{z_{3}}(\dx^{3}{\bf u_{\eps'}}))(x',t),$ for any $x'$ and $t.$
For the other term (the antiderivative of $f_{z_{2}}/f_{z_{3}}$) we simply again apply Lemma \ref{lem:coeff}.
These considerations yield the desired bound for $k_{G}.$

We still must estimate $\tilde{M}$ and $h.$  To begin with $\tilde{M},$ we must have an estimate for $a_{3}\in W^{3,\infty},$
for $a_{2}\in W^{2,\infty},$ for $a_{1}\in W^{1,\infty},$ and for $a_{0}\in W^{0,\infty}.$  We have already given exact
formulas for $a_{2}$ and $a_{3},$ so we will begin now with a description of $a_{1};$ this is in lieu of being fully
explicit with a formula for $a_{1}.$  We have decomposed $a_{1}$ previously as $a_{1}=b_{1}+b_{2},$ and we have given the
formula for $b_{2}$ in \eqref{b2Formula}.  For $b_{1},$ inspection of \eqref{bigThreeDerivatives}, together with the definition
of $b_{1},$ shows that the regularity of $b_{1}$ is like four derivatives of $u.$  Thus, $b_{2}$ is the most singular part, and
if we can bound $b_{2},$ then we have the requisite bound for $a_{1}.$  By Corollary \ref{2M},
each of $u_{\epsilon}$ and $u_{\epsilon'}$ are uniformly bounded in $H^{7};$ here, when we say
``uniformly,'' we refer to a bound independent of $\varepsilon$ or $\varepsilon',$ and also independent of $t.$
Together with assumption (A1), this implies that $b_{2}$ (and, as per our discussion, $a_{1}$ as well)  is bounded
in $H^{2}$ and thus in $W^{1,\infty},$ as desired.  The bounds in $W^{k,\infty}$ for the other coefficients $a_{k}$ are
similar, so we omit further details.

To complete our estimate of $\tilde{M},$ we must estimate the terms on the right-hand side of \eqref{linear-coeff}
which involve time derivatives.  For both of these terms, which are $\partial_{t}a_{3}$ and
$\partial_{t}\int_{0}^{x}\frac{a_{2}}{a_{3}}\ dy,$ that they are uniformly bounded in time in $L^{\infty}_{x}$ may be demonstrated
identically as in the proof of Lemma \ref{lem:coeff}.

All that remains is the estimate for $h$ in $L^{1}_{t}L^{2}_{x}.$  We can, in fact, bound $h$ in $L^{\infty}_{t}L^{2}_{x},$
and this clearly implies a bound in $L^{1}_{t}L^{2}_{x}$ over our finite time interval.  We treat $h_{4}$ differently from
$h_{j}$ for $j\in\{-1,0,1,2,3\}.$  From the definition of $h_{4},$ we see that
\begin{equation}\nonumber
\|h_{4}\|_{L^{\infty}_{t}L^{2}_{x}}\leq \varepsilon\sup_{t}\|\partial_{x}^{7}u_{\varepsilon'}\|_{L^{2}_{x}}
\leq \varepsilon M.
\end{equation}
We have given detailed formulas for $h_{3}$ and $h_{2}$ above.  From these formulas, it is clear that, because $f$ is smooth
and thus Lipschitz in its arguments, we have the following bound:
\begin{equation}\label{partsOf-h}
\|h_{2}+h_{3}\|_{L^{\infty}_{t}L^{2}_{x}}\leq C(M)\|u_{\varepsilon}-u_{\varepsilon'}\|_{L^{\infty}_{t}H^{3}_{x}}.
\end{equation}
While we have not written out $h_{-1},$ $h_{0},$ and $h_{1}$ fully explicitly, the completely analogous
estimate to \eqref{partsOf-h} is available for them.
Our conclusion for $h$ is then
\begin{equation}\nonumber
\|h\|_{L^{1}_{t}L^{2}_{x}}\leq C(M)T\|u_{\varepsilon}-u_{\varepsilon'}\|_{L^{\infty}_{t}H^3_{x}}
+\varepsilon MT.
\end{equation}
This completes the proof.
\end{proof}

\begin{prop} Let $\delta>0$ and $\delta'>0$ be given.  Let $\varepsilon$ and $\varepsilon'$ be given, such that
$0<\varepsilon<\varepsilon'.$  Let $u^{\varepsilon,\delta}$ and $u^{\varepsilon',\delta'}$ solve the evolution equation
\begin{equation}\nonumber
\partial_{t}v+f(\partial_{x}^{3}v,\ldots,v,x,t)=-\epsilon\partial_{x}^{4}v,
\end{equation}
on the common time interval $[0,T],$ with $T$ as above,
with parameter $\epsilon$ equal to $\varepsilon$ or $\varepsilon',$ respectively, and with initial data
$u^{\varepsilon,\delta}(\cdot,0)=(u_{0})_{\delta}$ and $u^{\varepsilon',\delta'}=(u_{0})_{\delta'},$
for given $u_{0}\in H^{7}.$  (Recall the definition of the regularized data
in \eqref{regularizationDefinition}.)
  Then
there exists $C>0,$ depending on $M,$ such that
\begin{equation}\label{convRateConclusion}
\|u^{\varepsilon,\delta}-u^{\varepsilon',\delta'}\|_{L^{\infty}_{t}H^{7}_{x}}
\leq \frac{C\varepsilon'}{{\delta'}^{4}}+\frac{C(o(\delta^{4})+o({\delta'}^{4}))}{{\delta'}^{4}}+o(1),
\end{equation}
where the $o(1)$ notation indicates a function which vanishes as $\delta\rightarrow 0$ and
$\delta'\rightarrow 0.$
\end{prop}

\begin{proof}
The proof of this proposition is similar to the proof of the previous proposition, but we estimate some terms differently
(making use of properties of mollifiers).
We let $w=\partial_{x}^{7}(u-u^{\prime}),$ and we write
\begin{equation}\nonumber
w_{t}=-\varepsilon\partial_{x}^{4}w+a_{3}\partial_{x}^{3}w+a_{2}\partial_{x}^{2}w+a_{1}\partial_{x}w+a_{0}w+h.
\end{equation}
We are again using Theorem \ref{thm:lin}, so we again need to check the bounds for the induced $k_{G}$ and
$\tilde{M}$ as defined in \eqref{linear-coeff}.  As we defined the coefficients earlier in the system \eqref{nonlinear-linear},
we again have the same formulas for the coefficients $a_{i},$ especially
\begin{equation}\nonumber
a_{3}=f_{z_{3}}(\dx^{3}{\bf u^{\eps,\delta}}),\qquad
a_{2}=f_{z_{2}}(\dx^{3}{\bf u^{\eps,\delta}})+7\partial_{x}\left(f_{z_{3}}(\dx^{3}{\bf u^{\eps,\delta}})\right).
\end{equation}
Together with the uniform bound of Corollary \ref{2M}, most of the required estimates for $k_{G}$ and $\tilde{M}$ are then
routine to check.  We focus now on the estimates for $w_{0}$ and $h.$

We do not provide full details, but instead focus on the most interesting terms.
To begin, we consider the initial data:
\begin{equation}\nonumber
\|w_{0}\|_{L^{2}}=\left\|\partial_{x}^{7}((u_{0})_{\delta}-(u_{0})_{\delta'})\right\|_{L^{2}}.
\end{equation}
We then use \eqref{eq:BS:L2} to bound this as
\begin{equation}\nonumber
\|w_{0}\|_{L^{2}}\leq o(1).
\end{equation}

Otherwise, we now focus on the two
most singular terms; we write
\begin{equation}\nonumber
h=h_{1}=h_{2}+h_{rest},
\end{equation}
with $h_{1}$ and $h_{2}$ defined by
\begin{equation}\nonumber
h_{1}=(\varepsilon-\varepsilon')\partial_{x}^{11}u^{\varepsilon',\delta'},
\end{equation}
and
\begin{equation}\nonumber
h_{2}=\left(\partial_{x}^{10}u^{\varepsilon',\delta'}\right)
\left(f_{z_{3}}(\dx^{3}{\bf u^{\eps,\delta}})-f_{z_{3}}(\dx^{3}{\bf u^{\eps',\delta'}})\right).
\end{equation}

Recalling that we must bound $h$ in $L^{1}_{t}L^{2}_{x},$ we estimate $h_{1}$ and $h_{2}$ in this space.
We begin with $h_{1};$ we find the following by making use of Corollary \ref{coro:remainder}:
\begin{equation}\nonumber
\|h_{1}\|_{L^{1}_{t}L^{2}_{x}}\leq T\|h_{1}\|_{L^{\infty}_{t}L^{2}_{x}}\leq T\max\{\varepsilon,\varepsilon'\}
\|u^{\varepsilon',\delta'}\|_{L^{\infty}_{t}H^{11}_{x}}
\leq \frac{C\max\{\varepsilon,\varepsilon'\}}{{\delta'}^{4}}=\frac{C\varepsilon'}{{\delta'}^{4}}.
\end{equation}

We next consider $h_{2},$ and begin as we did for $h_{1}:$
\begin{equation}\nonumber
\|h_{2}\|_{L^{1}_{t}L^{2}_{x}}\leq T \|h_{2}\|_{L^{\infty}_{t}L^{2}_{x}}\leq T
\|\partial_{x}^{10}u^{\varepsilon',\delta'}\|_{L^{\infty}_{t}L^{\infty}_{x}}
\left\|f_{z_{3}}(\dx^{3}{\bf u^{\eps,\delta}})-f_{z_{3}}(\dx^{3}{\bf u^{\eps',\delta'}})\right\|_{L^{\infty}_{t}L^{2}_{x}}.
\end{equation}
We use Sobolev embedding and we again use Corollary \ref{coro:remainder}, finding the following:
\begin{equation}\nonumber
\|h_{2}\|_{L^{1}_{t}L^{2}_{x}}\leq \frac{C}{\delta'^{4}}
\left\|f_{z_{3}}(\dx^{3}{\bf u^{\eps,\delta}})-f_{z_{3}}(\dx^{3}{\bf u^{\eps',\delta'}})\right\|_{L^{\infty}_{t}L^{2}_{x}}.
\end{equation}
A Lipschitz estimate for $f$ (as in the proof of Proposition \ref{lem:diff}) allows us to bound this as follows:
\begin{equation}\nonumber
\|h_{2}\|_{L^{1}_{t}L^{2}_{x}}\leq \frac{C}{{\delta'}^{4}}\|u^{\varepsilon,\delta}-u^{\varepsilon',\delta'}\|_{L^{\infty}_{t}H^{3}_{x}}.
\end{equation}
We use Proposition \ref{lem:diff}, and in particular \eqref{diff:lim}, to bound the right-hand side in terms of the initial data:
\begin{equation}\label{stillNeedDataDelta}
\|h_{2}\|_{L^{1}_{t}L^{2}_{x}}\leq \frac{C\varepsilon'}{{\delta'}^{4}}
+\frac{C}{{\delta'}^{4}}\|(u_{0})_{\delta}-(u_{0})_{\delta'}\|_{H^{3}_{x}}.
\end{equation}
Interpolating in \eqref{eq:BS:L2}, we have
\begin{equation}\label{interpMollStuff}
\|(u_{0})_{\delta}-(u_{0})_{\delta'}\|_{H^{3}}\leq C\left(o(\delta^{4})+o(\delta'^{4})\right).
\end{equation}

Using \eqref{interpMollStuff}  with \eqref{stillNeedDataDelta}, we conclude
\begin{equation}\nonumber
\|h_{2}\|_{L^{1}_{t}L^{2}_{x}}\leq \frac{C\varepsilon'}{{\delta'}^{4}}+\frac{C(o(\delta^{4})+o({\delta'}^{4}))}{{\delta'}^{4}}.
\end{equation}

Since $h_{1}$ and $h_{2}$ are the most singular terms, we omit the remaining details and consider the proof to be complete.
\end{proof}
The following corollary is an immediate consequence.

\begin{coro} \label{cauchyCorollary}
Let $\delta>0$ and $\delta'>0$ be given such that $\delta'>\delta.$
Let $\varepsilon=\delta^{5}$ and $\varepsilon'=\delta'^{5}.$
Let $u^{\delta}$ and $u^{\delta'}$ solve the initial value problem
\eqref{parab:non} on the common time interval $[0,T],$ with $T$ as above,
with parameter values $(\varepsilon,\delta)$ and $(\varepsilon',\delta'),$ respectively.  Then
\begin{equation}\nonumber
\|u^{\delta}-u^{\delta'}\|_{L^{\infty}_{t}H^{7}_{x}}
=o(1),
\end{equation}
where the $o(1)$ notation indicates a function which vanishes as $\delta'\rightarrow 0.$  This convergence
is uniform with respect to choice of $u_{0}$ if $u_{0}$ is taken from a compact set.
\end{coro}

\begin{proof}  The only statement which requires justification is the final statement, about uniformity of the convergence
when the initial data is taken from a compact set.  This uniformity is provided by Lemma \ref{BS}, for the term on the
right-hand side of \eqref{convRateConclusion} denoted as $o(1).$
\end{proof}

We are now able to conclude that our original system, \eqref{generalEquation}, has a solution in $H^{7}.$  This is the first
conclusion of Theorem \ref{main-thm}.
\begin{coro}\label{existencePartOfMainTheorem}
Under the assumptions of Theorem \ref{main-thm}, there exists $T=T(\|u_{0}\|_{H^{7}},\lambda_{0})$ such that
there exists a classical solution $u\in C_{[-T,T]}H^{7}$ of $\eqref{generalEquation}$ with $u(\cdot,0)=u_{0}.$
\end{coro}
\begin{proof}
From Corollary \ref{cauchyCorollary}, the sequence $u^{\delta}$ is Cauchy.  It has a limit, $u\in C_{[-T,T]}H^{7}.$
Integrating \eqref{parab:non} with respect to time, choosing $\varepsilon=\delta^{5}$ as in Corollary \ref{cauchyCorollary},
and using the Fundamental Theorem of Calculus, we get a formula for $u^{\delta}$ which involves at most fourth derivatives
of itself inside a time integral:
\begin{equation}\label{regIntegralRep}
u^{\delta}(\cdot,t)=(u_{0})_{\delta}+\int_{0}^{t}f(\dx^{3}{\bf u}^{\delta}(\cdot,\tau)) +\delta^{5}\partial_{x}^{4}u^{\delta}(\cdot,\tau)\ d\tau.
\end{equation}

  The convergence of $u^{\delta}$ to $u$ in $C_{[-T,T]}H^{7}$ implies that the
convergence of up to fourth derivatives is uniform in $t,$ and this uniform convergence allows us to pass to the limit
as $\delta\rightarrow0$ in the integral representation formula \eqref{regIntegralRep}.
After taking the limit, we take the time derivative again to find that
\eqref{generalEquation} is satisfied.
\end{proof}

\begin{rem} \label{uniformityRemark}
Again, the convergence of $u^{\delta}$ to $u$ is uniform with respect to the choice of initial data $u_{0}$
if this data is chosen from a compact subset of $H^{7}.$  To state this more precisely, let $\mathcal{K}$ be a compact
subset of $H^{7}.$  Recall that the time of existence of solutions to the initial value problem, $T,$
guaranteed by our existence theorem can be taken to be independent of the initial data $u_{0}\in\mathcal{K}.$
Let $\eta>0$ be given.  There exists $D>0$ such that for all $\delta\in(0,D),$
for all $u_{0}\in\mathcal{K},$
\begin{equation}\nonumber
\sup_{t\in[-T,T]}\|u^{\delta}(\cdot,t)-u(\cdot,t)\|_{H^{7}}<\eta,
\end{equation}
where $u^{\delta}$ and $u$ are the solutions of the unregularized and regularized problems, respectively, corresponding
to the unregularized initial data $u_{0}.$  This uniformity on compact sets stems from the uniformity in Lemma \ref{BS},
as in the proof of Corollary \ref{cauchyCorollary}.
\end{rem}

We are now able to state our continuous dependence result for initial data in $H^{7}.$
\begin{coro}
Let $u_{0}\in H^{7}$ and let $u_{n}$ be a sequence in $H^{7}$ such that $u_{n}\rightarrow u_{0}.$
Note that since $\{u_{n}\in H^{7}:n\in\{0,1,2,\ldots\}\}$ is compact in $H^{7}$,
the time of existence of solutions guaranteed by Corollary
\ref{existencePartOfMainTheorem} can be taken to be independent of $n;$ we therefore let
$U_{n}\in C_{[-T,T]}H^{7}$ be the solution of the initial value problem \eqref{generalEquation} with data
 $u_{n},$ for all $n\in\{0,1,2,\ldots\},$ with this $T$ independent of $n.$  Then,
 \begin{equation}\nonumber
 \lim_{n\rightarrow\infty}\sup_{t\in[-T,T]}\|U_{n}-U_{0}\|_{H^{7}}=0.
 \end{equation}
\end{coro}
\begin{proof}  Let $\eta>0$ be given.  Given any $\delta>0,$
we begin by adding and subtracting:
\begin{equation}\nonumber
U_{n}-U_{0}=\left(U_{n}-U_{n}^{\delta}\right)
+\left(U_{n}^{\delta}-U_{0}^{\delta}\right)
+\left(U_{0}^{\delta}-U_{0}\right).
\end{equation}
As in Remark \ref{uniformityRemark}, we may take a particular $\delta>0$ such that for all $m\in\{0,1,2,\ldots\},$
\begin{equation}\nonumber
\sup_{t\in[-T,T]}\|U_{m}^{\delta}(\cdot,t)-U_{m}(\cdot,t)\|_{H^{7}}<\frac{\eta}{3}.
\end{equation}
To make $U_{n}-U_{0}$ small, then, it is only necessary to focus on the difference $U_{n}^{\delta}-U_{0}^{\delta}.$
These are solutions of the approximate problem \eqref{parab:non} for fixed parameter $\delta.$  This is a parabolic problem,
and as such, has the continuous dependence result of Proposition \ref{parab:prop}.
Therefore there exists $N\in\mathbb{N}$ such that for all $n>N,$
we have
\begin{equation}\nonumber
\sup_{t\in[-T,T]}\|U_{n}^{\delta}(\cdot,t)-U_{0}^{\delta}(\cdot,t)\|_{H^{7}}<\frac{\eta}{3}.
\end{equation}
This completes the proof.
\end{proof}

\section{Fully explicit calculation of third derivative}

Here, we give a complete calculation of $\partial_{x}f,$ $\partial_{xx}f,$ and $\partial_{xxx}f.$
We begin with simply the first derivative:
\begin{equation}\nonumber
u_{xt}=f_{x}+f_{z_{0}}u_{x}+f_{z_{1}}u_{xx}+f_{z_{2}}u_{xxx}+f_{z_{3}}(\partial_{x}^{4}u).
\end{equation}
We apply another derivative with respect to $x,$ finding the following:
\begin{multline}\nonumber
u_{xxt}=f_{xx}+2f_{xz_{0}}u_{x}+2f_{xz_{1}}u_{xx}+2f_{xz_{2}}u_{xxx}+2f_{xz_{3}}(\partial_{x}^{4}u)
+f_{z_{0}}u_{xx}+f_{z_{0}z_{0}}u_{x}^{2}
\\
+2f_{z_{0}z_{1}}u_{x}u_{xx}+2f_{z_{0}z_{2}}u_{x}u_{xxx}+2f_{z_{0}z_{3}}u_{x}(\partial_{x}^{4}u)
+f_{z_{1}}u_{xxx}+f_{z_{1}z_{1}}u_{xx}^{2}\\
+2f_{z_{1}z_{2}}u_{xx}u_{xxx}+2f_{z_{1}z_{3}}u_{xx}(\partial_{x}^{4}u)+f_{z_{2}}(\partial_{x}^{4}u)
+f_{z_{2}z_{2}}u_{xxx}^{2}+2f_{z_{2}z_{3}}u_{xxx}(\partial_{x}^{4}u)\\
+f_{z_{3}}(\partial_{x}^{5}u)+f_{z_{3}z_{3}}(\partial_{x}^{4}u)^{2}.
\end{multline}
We differentiate once more.  The formula for the third derivative uses 59 terms on the right-hand side:
\begin{multline}\label{bigThreeDerivatives}
u_{xxxt}=f_{xxx}+3f_{xxz_{0}}u_{x}+3f_{xxz_{1}}u_{xx}+3f_{xxz_{2}}u_{xxx}
+3f_{xxz_{3}}(\partial_{x}^{4}u)+3f_{xz_{0}}u_{xx}
\\
+3f_{xz_{0}z_{0}}u_{x}^{2}+6f_{xz_{0}z_{1}}u_{x}u_{xx}
+6f_{xz_{0}z_{2}}u_{x}u_{xxx}+6f_{xz_{0}z_{3}}u_{x}(\partial_{x}^{4}u)
+3f_{xz_{1}}u_{xxx}\\
+3f_{xz_{1}z_{1}}u_{xx}^{2}+6f_{xz_{1}z_{2}}u_{xx}u_{xxx}
+6f_{xz_{1}z_{3}}u_{xx}(\partial_{x}^{4}u)+3f_{xz_{2}}(\partial_{x}^{4}u)
+3f_{xz_{2}z_{2}}u_{xxx}^{2}
\\
+6f_{xz_{2}z_{3}}u_{xxx}(\partial_{x}^{4}u)+3f_{xz_{3}}(\partial_{x}^{5}u)
+3f_{xz_{3}z_{3}}(\partial_{x}^{4}u)^{2}+f_{z_{0}}u_{xxx}
+3f_{z_{0}z_{0}}u_{x}u_{xx}\\
+3f_{z_{0}z_{1}}(u_{x}u_{xxx}+u_{xx}^{2})
+3f_{z_{0}z_{2}}(u_{x}(\partial_{x}^{4}u)+u_{xx}u_{xxx})
+3f_{z_{0}z_{3}}(u_{x}(\partial_{x}^{5}u)+u_{xx}(\partial_{x}^{4}u))\\
+f_{z_{0}z_{0}z_{0}}u_{x}^{3}
+3f_{z_{0}z_{0}z_{1}}u_{x}^{2}u_{xx}
+3f_{z_{0}z_{0}z_{2}}u_{x}^{2}u_{xxx}
+3f_{z_{0}z_{0}z_{3}}u_{x}^{2}(\partial_{x}^{4}u)
+3f_{z_{0}z_{1}z_{1}}u_{x}u_{xx}^{2}\\
+6f_{z_{0}z_{1}z_{2}}u_{x}u_{xx}u_{xxx}
+6f_{z_{0}z_{1}z_{3}}u_{x}u_{xx}(\partial_{x}^{4}u)
+3f_{z_{0}z_{2}z_{2}}u_{x}u_{xxx}^{2}
+6f_{z_{0}z_{2}z_{3}}u_{x}u_{xxx}(\partial_{x}^{4}u)
\\
+3f_{z_{0}z_{3}z_{3}}u_{x}(\partial_{x}^{4}u)^{2}
+f_{z_{1}}(\partial_{x}^{4}u)
+3f_{z_{1}z_{1}}u_{xx}u_{xxx}
+3f_{z_{1}z_{2}}(u_{xx}(\partial_{x}^{4}u)+u_{xxx}^{3})
\\
+3f_{z_{1}z_{3}}(u_{xx}(\partial_{x}^{5}u)+u_{xxx}(\partial_{x}^{4}u))
+f_{z_{1}z_{1}z_{1}}u_{xx}^{3}
+3f_{z_{1}z_{1}z_{2}}u_{xx}^{2}u_{xxx}
+3f_{z_{1}z_{1}z_{3}}u_{xx}^{2}(\partial_{x}^{4}u)\\
+3f_{z_{1}z_{2}z_{2}}u_{xx}u_{xxx}^{2}
+6f_{z_{1}z_{2}z_{3}}u_{xx}u_{xxx}(\partial_{x}^{4}u)
+3f_{z_{1}z_{3}z_{3}}u_{xx}(\partial_{x}^{4}u)^{2}
+f_{z_{2}}(\partial_{x}^{5}u)\\
+3f_{z_{2}z_{2}}u_{xxx}(\partial_{x}^{4}u)
+3f_{z_{2}z_{3}}(u_{xxx}(\partial_{x}^{5}u)+(\partial_{x}^{4}u)^{2})
+f_{z_{2}z_{2}z_{2}}u_{xxx}^{3}
+3f_{z_{2}z_{2}z_{3}}u_{xxx}^{2}(\partial_{x}^{4}u)\\
+3f_{z_{2}z_{3}z_{3}}u_{xxx}(\partial_{x}^{4}u)^{2}
+f_{z_{3}}(\partial_{x}^{6}u)+3f_{z_{3}z_{3}}(\partial_{x}^{4}u)(\partial_{x}^{5}u)
+f_{z_{3}z_{3}z_{3}}(\partial_{x}^{4}u)^{3}.
\end{multline}

\bibliographystyle{alpha}
\bibliography{kdv}
\end{document}